\documentclass{article}

\usepackage{blindtext, xcolor}

\usepackage[centertags]{amsmath}
\usepackage{hyperref}
\usepackage{amsfonts}
\usepackage{amssymb}
\usepackage{amsthm}
\usepackage{newlfont}
\usepackage{amscd}
\usepackage{amsmath,amscd}
\usepackage{graphicx}
\usepackage{ tipa }
\usepackage{mathrsfs}
\usepackage{comment}

\usepackage[all]{xy}
\usepackage{verbatim}
\usepackage{diagrams}

\usepackage{color}

\newcommand{\barr}{\overline}
\newcommand{\adj}{\rightleftarrows}

\newcommand{\CC}{\mathbb{C}}
\newcommand{\SSS}{\mathbb{S}}

\newcommand{\RR}{\mathbb{R}}
\newcommand{\NN}{\mathbb{N}}

\newcommand{\tC}{\mathrm{C}}

\newcommand{\cA}{\mathcal{A}}
\newcommand{\cB}{\mathcal{B}}
\newcommand{\cC}{\mathcal{C}}
\newcommand{\cD}{\mathcal{D}}

\newcommand{\cF}{\mathcal{F}}

\newcommand{\cL}{\mathcal{L}}
\newcommand{\cM}{\mathcal{M}}
\newcommand{\cN}{\mathcal{N}}
\newcommand{\cO}{\mathcal{O}}
\newcommand{\cP}{\mathcal{P}}

\newcommand{\cR}{\mathcal{R}}
\newcommand{\cS}{\mathcal{S}}

\newcommand{\cW}{\mathcal{W}}

\def\Csep{\mathtt{SC^*}}

\def\Csepi{\mathtt{SC_\infty^*}}
\def\CCsep{\mathtt{CSC^*}}
\def\FCW{\mathtt{CW}^f_*}

\def\NC{\mathtt{NC}}
\def\NCW{\mathtt{NCW}}
\def\Sp{\mathtt{Sp}}
\def\SpM{\mathtt{Sp^M}}
\def\Cat{\mathtt{Cat}}

\def\CSS{\mathtt{CSS}}
\def\sS{\mathcal{S}_{\mathtt{M}}}
\def\LMod{\mathtt{LMod}}
\def\RMod{\mathtt{RMod}}
\def\Ass{\mathtt{Ass}}
\def\Com{\mathtt{Com}}
\def\LM{\mathtt{LM}}

\def\Catc{\mathtt{Cat^L}}
\def\MC{\mathtt{ModCat}}
\def\SW{\mathtt{SW}}
\def\N{\mathtt{N}}
\def\M{\mathtt{M}}
\def\P{\mathtt{P}}
\def\Q{\mathtt{Q}}

\def\NSp{\mathtt{NSp}}

\def\FNCW{\mathtt{NCW}^f}

\def\CMp{\mathtt{CM_*}}
\newcommand{\Fin}{\mathtt{Fin}_*}
\newcommand{\NFin}{\mathtt{NFin_*}}

\def\Fins1{\mathtt{Fin^{\leq 1}_*}}

\def\Sphere{{\mathbb S}}

\def\Funfinc{\mathtt{Fun_*^{fc}}}

\def\Funl{\mathtt{Fun^L}}

\def\Lin{\cL}

\newtheorem{thm}{Theorem}[section]

\newtheorem{cor}[thm]{Corollary}

\newtheorem{lem}[thm]{Lemma}
\newtheorem{prop}[thm]{Proposition}

\theoremstyle{definition}
\newtheorem{define}[thm]{Definition}
\theoremstyle{remark}
\newtheorem{rem}[thm]{Remark}

\DeclareMathOperator{\C}{C}

\DeclareMathOperator{\id}{Id}

\DeclareMathOperator{\Fun}{Fun}

\DeclareMathOperator{\op}{op}

\DeclareMathOperator{\Set}{Set}

\DeclareMathOperator{\Ind}{Ind}
\DeclareMathOperator{\Map}{Map}
\DeclareMathOperator{\Alg}{Alg}

\DeclareMathOperator{\LKan}{LKan}
\DeclareMathOperator{\Quiv}{Quiv}

\DeclareMathOperator{\Hom}{Hom}

\DeclareMathOperator{\prehocolim}{hocolim}
\DeclareMathOperator{\precolim}{colim}

\DeclareMathOperator{\Top}{Top}

\def\colim{\mathop{\precolim}}
\def\hocolim{\mathop{\prehocolim}}

\def\lrar{\longrightarrow}

\DeclareFontEncoding{OT2}{}{} 
\DeclareTextFontCommand{\textcyr}{\fontencoding{OT2}\fontfamily{wncyr}\fontseries{m}\fontshape{n}\selectfont}

\newcommand\noloc{%
  \nobreak
  \mspace{6mu plus 1mu}
  {:}
  \nonscript\mkern-\thinmuskip
  \mathpunct{}
  \mspace{2mu}
}

\begin{document}
\title{Noncommutative CW-spectra as enriched presheaves on matrix algebras}

\author{Gregory Arone \thanks{Supported in part by the Swedish Research Council, grant number 2016-05440} \\ Stockholm University \\ gregory.arone@math.su.se \and Ilan Barnea \thanks{Supported by ISF 786/19} \\ Haifa University \\ ilanbarnea770@gmail.com \and Tomer M. Schlank \thanks{Supported by ISF 1588/18 and BSF 2018389} \\ Hebrew University \\tomer.schlank@gmail.com}

\maketitle

\begin{abstract}

Motivated by the philosophy that $C^*$-algebras reflect noncommutative topology, we investigate the stable homotopy theory of the (opposite) category of $C^*$-algebras. We focus on $\C^*$-algebras which are non-commutative CW-complexes in the sense of \cite{ELP}. We construct the stable $\infty$-category of noncommutative CW-spectra, which we denote by $\NSp$. Let $\cM$ be the full spectral subcategory of $\NSp$ spanned by ``noncommutative suspension spectra'' of matrix algebras. Our main result is that $\NSp$ is equivalent to the $\infty$-category of spectral presheaves on $\cM$.

To prove this we first prove a general result which states that any compactly generated stable $\infty$-category is naturally equivalent to the $\infty$-category of spectral presheaves on a full spectral subcategory spanned by a set of compact generators. This is an $\infty$-categorical version of a result by Schwede and Shipley \cite{ScSh1}. In proving this we use the language of enriched $\infty$-categories as developed by Hinich~\cite{Hin2,Hin3}.

We end by presenting a ``strict'' model for $\cM$. That is, we define a category $\cM_s$ strictly enriched in a certain monoidal model category of spectra $\SpM$. We give a direct proof that the category of $\SpM$-enriched presheaves $\cM_s^{\op}\to\SpM$ with the projective model structure models $\NSp$ and conclude that $\cM_s$ is a strict model for $\cM$.
\end{abstract}

\tableofcontents

\section{Introduction}
The celebrated   Gelfand  theorem gives a contravariant equivalence between the categories of locally compact Hausdorff spaces and commutative $C^*$-algebras. This correspondence led to the point of view that $C^*$-algebras are noncommutative generalizations of topological spaces. The study of $C^*$-algebras from this perspective is the subject of noncommutative geometry and topology. In this paper we study noncommutative stable homotopy theory, i.e., the stable homotopy category of the opposite of the category of $C^*$-algebras. In doing this we are continuing the investigations of \O stv\ae r~\cite{Ost} and Mahanta~\cite{Mah1}, among others.

To be a little more specific, in this paper we construct the $\infty$-category of noncommutative CW-spectra, which we denote $\NSp$, and show that $\NSp$ is equivalent to the category of spectral presheaves over a spectrally enriched category $\cM$. The objects of $\cM$ are noncommutative suspension spectra of matrix algebras, and its morphisms are mapping spectra between matrix algebras. In a companion paper~\cite{ABS2} we analyze the category $\cM$ in considerable detail. In that paper we introduce a rank filtration of $\cM$, describe the subquotients of the rank filtration, and use it, in particular, to give an explicit model for the rationalization of $\NSp$.

We construct the $\infty$-category of noncommutative CW-spectra, as the stabilization of the $\infty$-category of noncommutative CW-complexes. Our construction of the $\infty$-category of noncommutative CW-complexes mimics Lurie's construction of the $\infty$-category of ``ordinary'' CW-complexes as the ind-completion of the $\infty$-category of finite CW-complexes~\cite{Lur1}. The latter is considered an $\infty$-category by first viewing it as a topological category in the obvious way and then taking the topological nerve \cite[Section 1.1.5]{Lur1}. (By a topological category in this paper we mean a category enriched in the category of compactly generated weak Hausdorff spaces.)

In more detail, consider the class of $C^*$-algebras called ``noncommutative
CW-complexes'' in \cite[Section 2.4]{ELP}. These are algebras generated from the finite dimensional matrix algebras by a finite inductive procedure, generalizing the construction of \textbf{finite} CW-complexes from $S^0$ in the commutative case.
We will therefore call them \emph{finite noncommutative CW-complexes} in this paper.

Finite noncommutative CW-complexes have been studied in several places (for instance \cite{Ped} and \cite{Die}). In Section~\ref{ss:complexes} we define the \emph{topological category of finite noncommutative CW-complexes} to be the \textbf{opposite} of the topological category whose objects are the $C^*$-algebras which are noncommutative CW-complexes and whose hom-spaces are given by taking the topology of pointwise norm convergence on the sets of $*$-homomorphisms. We define the \emph{$\infty$-category of finite noncommutative CW-complexes} by taking the topological nerve of this topological category. We denote both versions of this category by $\FNCW$. We now define the \emph{$\infty$-category of noncommutative CW-complexes} to be the ind-completion of $\FNCW$. This will be our generalization of the $\infty$-category of spaces and we denote it by $\NCW$. It can be shown that the $\infty$-category $\FNCW$ is pointed, essentially small and admits finite colimits, so $\NCW$ is a pointed compactly generated $\infty$-category.

Let $\NSp:=\Sp(\NCW)$ be the \emph{$\infty$-category of noncommutative CW-spectra}, i.e., the stabilization of $\NCW$. By construction $\NSp$ is a stable $\infty$-category. In particular, it is enriched and tensored over the $\infty$-category of ``ordinary'' spectra $\Sp$. There is a suspension-spectrum functor from noncommutative spaces to noncommutative spectra, which we denote $\Sigma^\infty_{\NC}\colon \NCW\to \NSp$.
It can be shown (see Section \ref{ss:complexes}) that the (maximal) tensor product of $C^*$-algebras induces a closed symmetric monoidal structure on both $\NCW$ and $\NSp$, such that $\Sigma_{\NC}^\infty:\NCW\lrar\NSp$ is symmetric monoidal. Our main result is a presentation of $\NSp$ as a category of spectral presheaves over a full spectral subcategory spanned by an explicit set of generators.

In order to prove this we first prove a general result about presenting a compactly generated stable $\infty$-category as a category of spectral presheaves over a full spectral subcategory spanned by a set of compact generators. Such a result was proven by Schwede and Shipley \cite{ScSh1} using model categories (see also \cite{GM} for the same result under more general hypotheses). However, in this paper we need a more general result formed in the language of $\infty$-categories. To obtain this we use the formalism of enriched $\infty$-categories developed by Hinich \cite{Hin2,Hin3} and reviewed in Section \ref{ss:enriched}. We can formulate our result as follows:
\begin{thm}[Theorem \ref{t:modules}]\label{t:modules0}
  Let $\cD$ be a cocomplete stable $\infty$-category. Suppose that there is a small set $C$ of compact objects in $\cD$, that generates $\cD$ under colimits and desuspensions. Thinking of $\cD$ as left-tensored over the $\infty$-category of spectra $\Sp$, we let $\cC$ be the full $\Sp$-enriched subcategory of $\cD$ spanned by $C$. Then $\cD$ is naturally equivalent to the $\infty$-category of spectral presheaves on $\cC$, denoted $P_{\Sp}(\cC)$.
\end{thm}
There is also a monoidal version of the Theorem \ref{t:modules0}, given in Theorem \ref{t:modules monoidal}.

The classical stable infinity category of spectra $\Sp$ is generated by a single object, the sphere spectrum $\Sphere$, and this is closely related to the fact that $\Sp$ can be identified with the category of $\Sphere$-modules. By contrast $\NSp$ requires infinitely many generators. Let $M_n$ be the algebra of $n\times n$ matrices over $\mathbb{C}$. The algebras $\{M_n\mid n=1, 2, \ldots\}$ are the finite-dimensional simple $C^*$-algebras. Collectively, they play the same role in $\NCW$ as $S^0$ in the usual category of CW-complexes. The suspension spectra $\{\Sigma^\infty_{\NC} M_n\mid n=1, 2, \ldots \}$ are compact objects of $\NSp$, and they generate $\NSp$ under $\infty$-colimits and desuspensions. Let $\cM$ be the full $\Sp$-enriched subcategory of $\NSp$ spanned by $\{\Sigma^\infty_{\NC} M_n\mid n=1, 2, \ldots \}$.

For every $n,m\geq 0$ we have $M_n\otimes M_m\simeq M_{n\times m}$, so the set $\{M_n\mid n=1, 2, \ldots\}$ is closed under the tensor product. Since $\Sigma^\infty_{\NC}$ is monoidal, we see that $\{\Sigma^\infty_{\NC} M_n\mid n=1, 2, \ldots \}$ is also closed under the tensor product. It follows that the $\Sp$-enriched category $\cM$ acquires a symmetric monoidal structure from $\NSp$. This monoidal structure induces a symmetric monoidal structure on the $\infty$-category of spectral presheaves $P_\Sp (\cM)$ (Day convolution). Using the monoidal version of Theorem \ref{t:modules0} we obtain
\begin{thm}[Theorem \ref{theorem: main presentation}]\label{theorem: presentation0}
  The symmetric monoidal $\infty$-category $\NSp$ is naturally equivalent to the symmetric monoidal $\infty$-category $P_{\Sp}(\cM)$ of spectral presheaves on $\cM$.
\end{thm}

Thus, understanding the spectral $\infty$-category $\cM$ should help us understand the $\infty$-category $\NSp$. The objects of $\cM$ are in one-one correspondence with natural numbers and the monoidal product acts as multiplication. Given natural numbers $k, l$, we denote the corresponding mapping spectrum by $\SSS^{k,l}$
\[\SSS^{k,l}:=\Hom_{\NSp}(\Sigma^\infty_{\NC} M_k, \Sigma^\infty_{\NC} M_l).\]
One can describe $\SSS^{k,l}$ explicitly as follows. First, let us define a functor $G_{k,l}$ from finite pointed spaces to pointed spaces by the formula
\[
G_{k,l}(X)=\Map_{\FNCW}(M_k, X\wedge M_l).
\]
Since the pointed $\infty$-category $\FNCW$ has finite colimits, it is tensored over finite spaces and enriched over spaces. The spectrum $\SSS^{k,l}$ is the {\it stabilization} of $G_{k,l}$, i.e., $\SSS^{k,l}$ is the spectrum given by the sequence $\{G_{k,l}(S^0), G_{k,l}(S^1), \ldots$\}. In the companion paper~\cite{ABS2} we undertake a detailed study of the spectra $\SSS^{k,l}$ and the structure of $\cM$.

We end the paper by constructing a ``strict" version of $\cM$.
Namely, let $\SpM$ be the category of continuous pointed functors from finite pointed CW-complexes to topological spaces, endowed with the stable model structure. This is a symmetric monoidal model category, that models the $\infty$-category of spectra \cite{Lyd,MMSS}. In Definition \ref{d:M_s}, we define a symmetric monoidal category, strictly enriched in $\SpM$, denoted $\cM_s$. We give a direct proof of the following, which can be considered a strict version of Theorem \ref{theorem: presentation0}:
\begin{thm}[Theorem \ref{theorem: main presentation strict}]\label{theorem: main presentation strict0}
The category of $\SpM$-enriched functors $\cM_s^{\op}\to\SpM$ with the projective model structure and Day convolution is a symmetric monoidal model category that models the symmetric monoidal $\infty$-category $\NSp$.
\end{thm}

In Definition \ref{d:enriched localization} we define the notion of enriched $\infty$-localization. This takes a category strictly enriched in a monoidal model category, and produces an $\infty$-category enriched in the $\infty$-localization of this model category (see also Remark \ref{r:coherent}). A consequence of Theorem \ref{theorem: main presentation strict0} is that the enriched $\infty$-localization of $\cM_s$ is equivalent to $\cM$.

\begin{rem}
Using the work of Blom and Moerdijk~\cite{BM}, it is possible to define a model category structure on the opposite of the pro-category of separable $C^*$-algebras, that models $\NCW$. This model structure is a right Bousfield localization of the model structure presented in \cite{BJM}. We might then be able to use known results on stable model categories to prove a similar result to Theorem \ref{theorem: main presentation strict0}. We did not pursue this approach in this paper.
\end{rem}

\subsubsection*{An alternative definition, via nonabelian derived categories}
We will now digress to describe another natural way of defining a noncommutative analogue to the $\infty$-category of pointed spaces of Lurie. It relies even more on $\infty$-categorical constructions, and we do not develop it in this paper. Namely, we can do so using the concept of a nonabelian derived category (see \cite[Section 5.5.8]{Lur1}). If $\cC$ is a small $\infty$-category with finite coproducts, Lurie defines the nonabelian derived category of $\cC$, denoted $\cP_\Sigma(\cC)$, as the $\infty$-category obtained from $\cC$ by formally adjoining colimits of sifted diagrams. Loosely speaking sifted diagrams are generated by filtered diagrams and the simplicial diagram $\Delta^{\op}$. Taking $\cC=\Fin$ to be the category of finite pointed sets, we obtain the $\infty$-category of pointed spaces, that is, we have a natural equivalence $\cP_\Sigma(\Fin)\simeq \Top$.

Under the  Gelfand  correspondence, the finite pointed sets correspond to the finite dimensional commutative $C^*$-algebras.
We thus denote by $\NFin$ the full subcategory of $\FNCW$ spanned by the finite dimensional $C^*$-algebras (which are just finite products of matrix algebras).
We can now define a noncommutative analogue to the $\infty$-category of pointed spaces to be the nonabelian derived category of $\NFin$:
$$\barr{\NCW}:=\cP_\Sigma(\NFin).$$
We have natural inclusions
$$\NFin\hookrightarrow\FNCW\hookrightarrow\Ind(\FNCW)=\NCW$$
and the $\infty$-category $\NCW$ admits sifted colimits, so by the universal property we have an induced functor
$$\barr{\NCW}=\cP_\Sigma(\NFin)\to\NCW,$$
that commutes with sifted colimits. We do not know if this functor is an equivalence. This is true iff for every $n\geq 0$ and every simplicial object $X$ in $\FNCW$ the natural map
$${\colim}_{\Delta^{\op}} \Map_{\NCW}(M_n, X)\to \Map_{\NCW}(M_n, {\colim}_{\Delta^{\op}}^{\NCW} X)$$
is an equivalence. We know how to prove this last assertion when $X$ is a simplicial object in $\NFin$ or $X$ has the form $Y\otimes M_k$ for $k\geq 1$ and $Y$ is a simplicial object in $\FNCW$ composed of commutative algebras.

What we do know is that the induced map on stabilizations:
$$\Sp(\barr{\NCW})\to\Sp(\NCW)=\NSp$$
is an equivalence. To see this note that, by a similar reasoning as in the beginning of Section \ref{s:NCWSp}, we have that  $\barr{M}:=\{\Sigma^\infty M_i\mid i\in \NN\}$ generates $\Sp(\barr{\NCW})$ under small colimits. Thus, by Theorem \ref{theorem: main presentation}, it is enough to show that
for every $k,l\geq 1$ the induced map
$$\Hom_{\Sp(\barr{\NCW})}(\Sigma^\infty M_k,\Sigma^\infty M_l)\to\Hom_{\NSp}(\Sigma^\infty M_k,\Sigma^\infty M_l)$$
is an equivalence. We can define the functor
$\barr{G}_{k,l}$ from finite pointed spaces to pointed spaces by
$$\barr{G}_{k,l}(X):=\Map_{\barr{\NCW}}(M_k,X\wedge M_l).$$
The stabilization of $\barr{G}_{k,l}$ is the mapping spectrum
$$\Map_{\Sp(\barr{\NCW})}
(\Sigma^\infty M_k, \Sigma^\infty M_l).$$
It is thus enough to show that the induced natural transformation $\barr{G}_{k,l}\to G_{k,l}$ is an equivalence.
The functor $\barr{G}_{k,l}$ clearly commutes with sifted colimits and therefore it is equivalent to the (derived) left Kan extension of $\barr{G}_{k,l}|_{\Fin}$ along the inclusion $\Fin\subseteq \cS_*$. We prove in~\cite{ABS2} that the same holds for the functor $G_{k,l}$. Therefore it is enough to show that the restriction $\barr{G}_{k,l}|_{\Fin}\to G_{k,l}|_{\Fin}$ is an equivalence. But for every $[t]\in \Fin$ we have
$$\barr{G}_{k,l}([t])=\Map_{\cP_\Sigma(\NFin)}(M_k,[t]\wedge M_l) \simeq\Map_{\NFin}(M_k,M_l^t) =$$
$$\Map_{\FNCW}(M_k,M_l^t) \simeq\Map_{\Ind(\FNCW)}(M_k,[t]\wedge M_l)  =G_{k,l}([t]),$$
so we are done. Note that since the main results in this paper (and in \cite{ABS2}) concern the stabilization of $\NCW$, they apply equally well to the stabilization of the alternative model $\barr{\NCW}$.

\subsubsection*{Comparison with previous work} We end the introduction by relating the $\infty$-categories $\FNCW$ and $\NCW$ constructed here with different $\infty$-categories constructed in \cite{Mah1}. For more detail see Section \ref{ss:complexes}. In \cite{Mah1}, Mahanta constructed the $\infty$-category $\Csepi$ as the topological nerve of the topological category of \textbf{all} separable $C^*$-algebras, with the mapping spaces given by the topology of pointwise norm convergence on the sets of $*$-homomorphisms. He called $(\Csepi)^{\op}$ the $\infty$-category of \emph{pointed
compact metrizable noncommutatives spaces}. He then defined the $\infty$-category $\N{\cS}_*$
as the ind-completion of $(\Csepi)^{\op}$, and called it the $\infty$-category of \emph{pointed noncommutative spaces}.

It can be shown that our $\infty$-category $\FNCW$ is a full subcategory of $(\Csepi)^{\op}$ and the inclusion commutes with finite colimits. It follows that our $\infty$-category $\NCW$ is a coreflective full subcategory of $\N\cS_*$, and thus the inclusion admits a right adjoint:
$$i\colon\NCW\adj\N\cS_*\noloc R.$$
We call a morphism $g\colon X\to Y$ in $\N\cS_*$ a \emph{weak homotopy equivalence} if for every $n\geq 1$ the induced map
     $$g_*\colon\Map_{\N\cS_*}(M_n,X)\to\Map_{\N\cS_*}(M_n,Y)$$
     is an equivalence of spaces.
This is analogous to weak homotopy equivalences between topological spaces.
The counit of the adjunction above $i\circ R \to \id_{\N\cS_*}$ is a levelwise weak equivalence, and thus can be thought of as a CW approximation to elements in $\N\cS_*$. If $X$ and $Y$ are noncommutative CW-complexes then $g$ is a weak equivalence iff it is an equivalence in $\N{\cS}_*$.

Informally speaking, since the equivalences in $(\Csepi)^{\op}$ are homotopy equivalences of $C^*$-algebras, the category $\N\cS_*=\Ind((\Csepi)^{\op})$ is somewhat analogous to the infinity category modeled by the Str\o m model structure on topological spaces \cite{Str}, in which the weak equivalences are the homotopy equivalences. The category $\NCW$ constructed here is analogous to the infinity category modeled by the Quillen model structure on topological spaces, in which the weak equivalences are the \textbf{weak} homotopy equivalences.

\subsubsection*{Section by section outline of the paper}
In Section~\ref{ss:complexes} we define the $\infty$-category $\NCW$: a noncommutative analogue of the $\infty$-category of pointed spaces.

In Section~\ref{ss:enriched} we review the theory of enriched $\infty$-categories, as developed by Hinich~\cite{Hin2, Hin3}. In particular we state the enriched Yoneda lemma for $\infty$-categories. We also present a way to pass from model categories to $\infty$-categories is the enriched setting.

In Section~\ref{s:stab_infinity} we review the notion of a stable $\infty$-category and show that a compactly generated stable $\infty$-category is equivalent to the category of spectral presheaves on a full subcategory spanned by a set of compact generators.

In Section~\ref{ss:stab} we review the process of stabilizing an $\infty$-category. We use the framework established by Lurie in~\cite[Section 1.4]{Lur2}. We also review how a similar procedure can be applied to an ordinary topologically enriched category, and compare the strict and the $\infty$-categorical versions of stabilization.

In Section~\ref{s:NCWSp} we define the category of non-commutative CW-spectra $\NSp$ as the stabilization of the category $\NCW$. We identify the  suspension spectra of matrix algebras as an explicit set of generators of $\NSp$. Letting $\cM$ be the full subcategory of $\NSp$ spanned by matrix algebras, we conclude that $\NSp$ is (monoidally) equivalent to $P_\Sp (\cM)$, the category of spectral presheaves on $\cM$. We give a strict model for $\cM$, denoted $\cM_s$, as a category enriched over a Quillen model category of spectra $\SpM$. We also show that the category of $\SpM$-enriched functors $\cM_s^{\op}\to\SpM$ with the projective model structure
models the $\infty$-category $\NSp$ and conclude that $\cM$ is equivalent to the enriched $\infty$-localization of $\cM_s$.

\subsubsection*{Acknowledgements}
We are grateful to Vladimir Hinich for explaining to us his theory of enriched infinity categories and its relevance to our work.

\section{The $\infty$-category of noncommutative CW-complexes}\label{ss:complexes}

In this section we define a noncommutative analogue of the $\infty$-category of pointed spaces defined by Lurie \cite{Lur1}.

Let $\Csep$ (resp. $\CCsep$) denote the category of all (resp. commutative) separable $C^*$-algebras and $*$-homomorphisms. Following the common convention in the field, the term $C^*$-algebra or $*$-homomorphism will always mean \emph{non-unital}. The  Gelfand  correspondence implies that the functor
$$X\mapsto\tC_0(X):\CMp\lrar\CCsep^{\op}$$
that assigns to every pointed compact metrizable space $X$ the commutative separable $C^{*}$-algebra of continuous complex valued functions on $X$ that vanish at the basepoint, is an equivalence of categories.
It is thus natural to regard $\Csep^{\op}$ as the category of \emph{noncommutative} pointed compact metrizable spaces.

Consider $\CMp$ as a topologically enriched category, where for every $X,Y\in\CMp$ we endow the set of pointed continuous maps $\CMp(X,Y)$ with the \emph{compact open topology}. Now we take the topological nerve \cite[Section 1.1.5]{Lur1} of this topological category and obtain the $\infty$-category $(\CMp)_\infty$.
It is well-known that $(\CMp)_\infty$ admits finite $\infty$-colimits and that $\infty$-pushouts can be calculated using the standard cylinder object.

Let us construct the $\infty$-category of pointed spaces in a way that admits a natural generalization to the noncommutative case.
We denote by $\FCW$ the smallest full subcategory of $\CMp$ that contains $S^0$ and is closed under finite homotopy-colimits using the standard cylinder object. Thus $\FCW$ is the topological category of finite pointed CW-complexes. We will also consider $\FCW$ as an $\infty$-category, by applying the coherent nerve functor to it. We will use the same notation $\FCW$ to indicate both the ordinary (topologically enriched) and the $\infty$-categorical incarnation of the category, trusting that it is clear from the context which is meant.
The $\infty$-category of pointed spaces can be defined as the $\infty$-categorical $\Ind$ construction of $\FCW$.
Note that under the  Gelfand  correspondence, $S^0$ corresponds to $\CC$, which is the only nonzero finite dimensional simple commutative $C^*$-algebra.

We now turn to the noncommutative analogue. We first recall  from \cite[Section 2.1]{Mah1} the construction of the $\infty$-category $\Csepi$.
Consider $\Csep$ as a topologically enriched category, where for every $A,B\in\Csep$ we endow the set of $*$-homomorphisms $\Csep(A,B)$ with the topology of pointwise norm convergence. Now we take the topological nerve of this topological category and obtain the $\infty$-category $\Csepi$. It is shown in \cite[Section 2.1]{Mah1} that $\Csepi$ is (essentially) small, pointed, and finitely complete.

\begin{rem}\label{r:coherent}
Recall that any relative category, that is a pair $(\cC,\cW)$ consisting of a category $\cC$ an a subcategory $\cW\subseteq\cC$, has a canonically associated $\infty$-category $\cC_\infty$, obtained by formally inverting the morphisms in $\cW$, in the infinity categorical sense. There is also a canonical localization functor $\cC\to\cC_\infty$ satisfying a universal property. We refer the reader to \cite{Hin1} for a thorough account, and also to the discussion in \cite[Section 2.2]{BHH}. We refer to $\cC_\infty$ as the $\infty$-localization of $\cC$ (with respect to $\cW$). If $\cC$ is a model category or a (co)fibration category, we always take $\cW$ to be the set of weak equivalences in $\cC$.

Using $\infty$-localization, there is another natural way of considering separable $C^*$-algebras as an $\infty$-category. There is a well known notion of homotopy equivalence between $C^*$-algebras. We can consider $\Csep$ as a relative category, with the weak equivalences given by the homotopy equivalences, and take its $\infty$-localization. This is the point of view taken, for instance, in \cite{AG,Uuy}. It follows from \cite[Proposition 3.17]{BJM} that we obtain an $\infty$-category equivalent to $\Csepi$.
\end{rem}

\begin{rem}\label{remark: cotensoring}
It is well-known that $\Csep$ is cotensored over the category of pointed finite CW-complexes~\cite{AG}. If $K$ is a finite pointed CW-complex and $A\in \Csep$ then the cotensoring of $A$ and $K$ is given by the $C^*$-algebra of pointed continuous functions from $K$ to $A$. One can define finite homotopy limits in $\Csep$ using this cotensoring. Consequently, the $\infty$-pullbacks in $\Csepi$ can be calculated as homotopy pullbacks using the standard path object~\cite[Proposition 2.7]{Mah1}.
\end{rem}

\begin{define}\label{definition: FNCW}
We denote by $\FNCW$ the opposite of the smallest full subcategory of $\Csep$ that contains the nonzero finite dimensional simple algebras in $\Csep$ (which are just the matrix algebras over $\CC$) and is closed under finite homotopy-limits using the standard path object. We call $\FNCW$ the category of \emph{finite pointed noncommutative CW-complexes}. The category $\FNCW$ is an ``ordinary'' topologically enriched category. We will also consider $\FNCW$ as an $\infty$-category, by applying the coherent nerve functor to it. Like in the commutative case, will use the same notation $\FNCW$ to indicate both the ordinary (topologically enriched) and the $\infty$-categorical incarnation of the category, trusting that it is clear from the context which is meant.
\end{define}
The topological category $\FNCW$ is the category of \emph{noncommutative CW-complexes} as defined in \cite[Section 2.4]{ELP}. Using \cite[Theorem 11.14]{Ped}, the same proof as in \cite[Proposition 1.1]{Mah2} gives that the (maximal) tensor product of $C^*$-algebras induces a symmetric monoidal structure on the $\infty$-category $\FNCW$, that preserves finite colimits in each variable separately. Note also that since the topological category $\Csep$ is cotensored over pointed finite CW-complexes, the topological category $\FNCW$ is tensored over pointed finite CW-complexes. We will denote the tensoring of a finite CW-complex $K$ and a noncommutative finite complex $X$ by $K\wedge X$.

We now define the $\infty$-category of \emph{noncommutative pointed CW-complexes} to be the $\infty$-categorical ind-completion of $\FNCW$,
$$\NCW:=\Ind(\FNCW).$$
The $\infty$-category $\FNCW$ is (essentially) small, pointed and finitely cocomplete so $\NCW$ is a compactly generated pointed $\infty$-category.  By \cite[Corollary 4.8.1.14]{Lur2} there is an induced closed symmetric monoidal structure on the $\infty$-category $\NCW\simeq\Ind(\FNCW)$ such that the natural embedding $j \colon \FNCW\lrar\NCW$ is symmetric monoidal.

In \cite{Mah1}, Mahanta defined the $\infty$-category $\N\cS_*$ as the ind-completion of $(\Csepi)^{\op}$, and called
 it the $\infty$-category of \emph{pointed noncommutative spaces}. By definition our category $\FNCW$ is a full subcategory of $(\Csepi)^{\op}$. Since $\infty$-pushouts in both $\FNCW$ and $(\Csepi)^{\op}$ can be calculated as homotopy pushouts using the standard cylinder object, we see that the inclusion $\FNCW\hookrightarrow(\Csepi)^{\op}$ commutes with finite colimits. Passing to ind-completions we get that the induced inclusion $\NCW\hookrightarrow\N\cS_*$ admits a right adjoint (or in other words, $\NCW$ is a coreflective full subcategory of $\N\cS_*$)
$$i\colon\NCW\adj\N\cS_*\noloc R.$$
We call a morphism $g\colon X\to Y$ in $\N\cS_*$ a \emph{weak homotopy equivalence} if $R(g)$ is an equivalence in $\NCW$, or equivalently, if for every object $W$ in $\NCW$ the induced map
     $$g_*\colon \Map_{\N\cS_*}(i(W),X)\to\Map_{\N\cS_*}(i(W),Y)$$
     is an equivalence in $\cS$.
     Since every object in $\NCW$ is a small colimit of matrix algebras and $i$ commutes with small colimits, we see that $g$ is a weak equivalence iff for every $n\geq 1$ the induced map
     $$g_*\colon \Map_{\N\cS_*}(M_n,X)\to\Map_{\N\cS_*}(M_n,Y)$$
     is an equivalence in $\cS$.
This is analogous to weak homotopy equivalences between topological spaces.
The counit of the adjunction above $i\circ R \to \id_{\N\cS_*}$ is a levelwise weak equivalence, and thus can be thought of as a CW approximation to elements in $\N\cS_*$. If $X$ and $Y$ are noncommutative CW-complexes then $g$ is a weak equivalence iff it is an equivalence in $\N\cS_*$.
Since the weak equivalences in $(\Csepi)^{\op}$ are the homotopy equivalences, the category $\N\cS_*$ is somewhat analogous to the infinity category modeled by the Str\o m model category on topological spaces \cite{Str}.

\section{Enriched infinity categories}\label{ss:enriched}

In Theorem \ref{t:modules0} we make use of enriched infinity categories. There are a few approaches to this theory (see, for example, \cite{GH,Lur2}) but so far only in \cite{Hin2} the Yoneda embedding is defined and its basic properties are shown. Since we need these results, we chose to follow Hinich's approach in this paper. In this section we give an overview of the basic definitions and constructions needed for later on. We also present some new material in subsection \ref{ss:enriched model}, concerning the connection between model categories and $\infty$-categories in the enriched setting.

Let $\Cat$ denote the $\infty$-category of $\infty$-categories, and let $\Catc$ denote the $\infty$-subcategory of $\Cat$ whose objects are $\infty$-categories having small colimits and whose morphisms preserve these colimits. The category $\Cat$ is symmetric monoidal under the cartesian product, while $\Catc$ has a symmetric monoidal structure induced from the cartesian structure on $\Cat$ (see \cite[4.8.1.3, 4.8.1.4]{Lur2}). With this structure on $\Catc$,
$\Map_{\Catc}(P\otimes L,M)$ is the subspace of $\Map_{\Cat}(P\times L,M)$ consisting of functors preserving small colimits along each argument. Note that a monoidal $\infty$-category is equivalent to an associative algebra object in $\Cat$, while an associative algebra in $\Catc$ is equivalent to a monoidal $\infty$-category with colimits, whose monoidal product commutes with colimits in each variable. We define a \emph{closed monoidal $\infty$-category} to be an associative algebra in $\Catc$.

If $\cM$ is a closed monoidal $\infty$-category, then a category left-tensored over $\cM$ is by definition a left module over $\cM$ in $\Catc$.
More generally, if $\cO$ is an $\infty$-operad, an $\cO$-monoidal category is an algebra over $\cO$ in $\Catc$. If $\cM$ is an $\cO$-monoidal category one can define an $\cO$-algebra in $\cM$.

Let $\Ass$ be the associative operad and $\LM$ be the two colored operad of left modules.
Algebras over $\Ass$ are associative algebras and algebras over $\LM$ consist of an associative algebra and a left module over it.
Thus, an $\Ass$-monoidal category is just a closed monoidal $\infty$-category and an $\LM$-monoidal category is a pair consisting of a closed monoidal $\infty$-category and a category left-tensored over it.

Let $\cM$ be a closed monoidal $\infty$-category. For every space (i.e., an $\infty$-groupoid) $X$, Hinich constructs (see \cite[Sections 3 and 4]{Hin2}) a closed monoidal structure on the $\infty$-category of Quivers
$$\Quiv_X(\cM):=\Fun(X^{\op}\times X,\cM).$$  Hinich's monoidal structure is an $\infty$-categorical version of the usual convolution product that one uses to define ordinary enriched categories.
For $\cB$ a category left-tensored over $\cM$, Hinich constructs a left action of the closed monoidal $\infty$-category $\Quiv_X(\cM)$ on the $\infty$-category
$\Fun(X,\cB)$. In his notation we obtain an $\LM$-monoidal category
$$\Quiv^{\LM}_X(\cM,\cB):=(\Quiv_X(\cM),\Fun(X,\cB)).$$
\begin{define}
A $\cM$-enriched category, with space of objects $X$ is an associative algebra in $\Quiv_X(\cM)$.
\end{define}
\begin{rem}
Hinich uses the term $\cM$-enriched {\it precategory} for an associative algebra in $\Quiv_X(\cM)$. He reserves the term $\cM$-enriched category for precategories satisfying a version of Segal completeness condition (see \cite[Definition 7.1.1]{Hin2}). In this paper we are not concerned with Segal completeness, so we will not distinguish between enriched categories and precategories. We will just say ``enriched category'' where Hinich might have said ``enriched precategory''. 
\end{rem}

\begin{rem}
If $\cM$ is the $\infty$-category of spaces, then the category of $\cM$-enriched categories with space of objects $X$ is equivalent to the category of simplicial spaces satisfying the Segal condition and equalling $X$ in simplicial degree zero. See~\cite[Corollary 5.6.1]{Hin2}, where a more general statement is proved. In other words, a category enriched in spaces is the same thing as an ordinary $\infty$-category. For a general closed monoidal $\infty$-category $\cM$, there is a monoidal ``forgetful'' functor from $\cM$ to spaces, given by $\Map_{\cM}(\mathtt{1}, -)$. In this way we obtain a forgetful functor from the category of $\cM$-enriched categories to ordinary infinity categories (compare with~\cite[Definition 7.1.1]{Hin2}).
\end{rem}

\begin{rem}\label{r:enriched model}
As we show in next subsection \ref{ss:enriched model}, the theory of monoidal model categories and categories enriched or tensored over them extends nicely to the theory presented above upon application of $\infty$-localization.
\end{rem}

In \cite[Section 6]{Hin2} Hinich defines the notion of an $\cM$-functor from an $\cM$-enriched category to a category left-tensored over $\cM$. Let $\cA$ be an $\cM$-enriched category with space of objects $X$ and $\cB$ a category left-tensored over $\cM$. Then $\cA$ is an associative algebra in $\Quiv_X(\cM)$ and $\Fun(X,\cB)$ is left tensored over $\Quiv_X(\cM)$. An $\cM$-functor $\cA\to\cB$ is defined to be a left module over $\cA$ in $\Fun(X,\cB)$, and the $\infty$-category of $\cM$-functors $\cA\to\cB$ is defined to be
$$\Fun_{\cM}(\cA,\cB):=\LMod_{\cA}(\Fun(X,\cB)).$$


For $x, y$ objects of $\cB$, we define the presheaf $\Hom_\cB(x,y)\in P(\cM)$ by
$$\Hom_\cB(x,y)(K):=\Map_\cB(K\otimes x,y).$$
Clearly $\Hom_\cB(x,y):\cM^{\op}\to\cS$ preserves limits, but it is not necessarily representable. If it happens to be representable, then the representing object serves as an internal mapping object from $b$ to $c$. Every $\cM$-functor $F:\cA\to\cB$ induces maps in $P(\cM)$
$$h_{\cA(x,y)}\to\Hom_\cB(F(x),F(y))$$
for $x,y\in X$.
The $\cM$-functor $F$ is called $\cM$-fully faithful if all these maps are equivalences.

If $\Hom_\cB(b,c):\cM^{\op}\to\cS$ is representable for all objects $b, c$ of $\cB$, then $\cB$ is enriched as well as left tensored. More generally and more precisely, Hinich proves the following proposition (see~\cite[Proposition 6.3.1 and Corollary 6.3.4]{Hin2}). We note that if $\cM$ is presentable, then any functor $\cM^{\op}\to\cS$ that preserves limits is representable (see \cite[Proposition 5.5.2.2]{Lur1}).

\begin{prop}\label{proposition: enrichment}
Let $\cM$ be a closed monoidal $\infty$-category and $\cB$ left tensored over $\cM$. Let $C$ be a class of objects of $\cB$. If for all $x, y\in C$, the presheaf $\Hom_{\cB}(x, y)$ is representable then there exists an $\cM$-enriched category $\cC$ whose class of objects is $C$, such that for any two objects $x, y$ of $C$, the morphism object $\cC(x, y)$ is a representing object for the functor $\Hom_{\cB}(x, y)$. There is
a fully-faithful $\cM$-enriched functor $\cC\to \cB$, extending the inclusion of $C$ into $\cB$.
\end{prop}

Using \cite[Lemma 6.3.3]{Hin2} it is not hard to see that given the conditions of Theorem \ref{proposition: enrichment} all the categories $\cC$ that can be obtained are canonically equivalent (via the choice of $X$ as the full subspace of $\cB$ spanned by $C$ in \cite[Corollary 6.3.4]{Hin2}). We can thus call $\cC$ \textbf{the} full enriched subcategory of $\cB$ spanned by $C$.
Taking $C$ to be the class of all objects in $\cB$ we see that if $\Hom_{\cB}(x, y)$ is representable for all $x, y$, then $\cB$ is enriched as well as tensored over $\cM$.

%


\subsubsection{From enriched model categories to enriched infinity categories}\label{ss:enriched model}

Let $\Cat_1$ denote the category of small categories and functors between them and let $\sS$ denote the category of simplicial sets. We have the usual nerve functor
$$\N:\Cat_1\to\sS.$$
The functor $\N$ is limit preserving and in particular it is a (cartesian) monoidal functor.

In \cite[Section 2.1]{Hor}, Horel constructs a (large, coloured) $\Cat_1$-operad denoted $\MC$. The colours in $\MC$ are model categories, while the category of multilinear operations $\Map_{\MC}(\cM_1,\cdots,\cM_n,\cN)$ is the category of left Quillen $n$-functors
$$\cM_1\times\dots\times\cM_n\to\cN$$
and natural weak equivalences (on cofibrant objects) between them. Since $\N$ is a monoidal functor, we obtain a simplicial operad from $\MC$ by composing with $\N$. We denote this simplicial operad also by $\MC$.

Let $\sS^{\Delta^{\op}}$ denote the category of simplicial objects in $\sS$ with Rezk's model structure. This is a combinatorial simplicial cartesian closed symmetric monoidal model category with all objects cofibrant. Let $\CSS$ denote the full simplicial subcategory of $\sS^{\Delta^{\op}}$ spanned by the fibrant objects. Then $\CSS$ is a monoidal simplicial category (under the cartesian product) whose simplicial nerve is naturally equivalent to the monoidal $\infty$-category $\Cat$.
Horel also constructs in [loc. cit.] another full simplicial subcategory $\Cat_\infty\subseteq\sS^{\Delta^{\op}}$ containing $\CSS$ and closed under the cartesian product. He shows that the inclusion $\CSS\to\Cat_\infty$ induces an equivalence of $\infty$-categories after application of simplicial nerve.

Let $\MC^c$ be the full sub simplicial operad of $\MC$ on model categories that are Quillen equivalent to a combinatorial model category.
For $\cM\in \MC^c$, let $\N_\cR(\cM)$ denote the Resk nerve construction on the cofibrant objects in $\cM$ and weak equivalences between them. By \cite[Theorem 2.16]{Hor}, $\N_\cR$ extends to a map of simplicial operads
$\MC^c\to \Cat_\infty$. Applying simplicial nerve, we obtain a map of $\infty$-operads
$\N_\cS(\MC^c)\to \Cat$. By \cite[Remark 2.17]{Hor}, this map factors through $\Catc$. Since Resk's nerve is one of the models for $\infty$-localization (see, for example, \cite[Section 2.2]{BHH}), we obtain a map of $\infty$-operads $$(-)_\infty:\N_\cS(\MC^c)\to\Catc$$
which acts as $\infty$-localization on objects.


Let $\M$ be the nonsymmetric operad (in $\Set$) freely generated by an operation in degree 0 and 2.
An algebra over $\M$ in $\Set$ is a set with a binary multiplication and a base point. Let $\P$ be the operad
in $\Cat_1$ which is given in degree $n$ by the groupoid whose objects are points of $\M(n)$ and a with a unique
morphism between any two objects. Then an algebra over $\P$ in $\Cat_1$ is a monoidal category.
The nerve of $\P$ is a simplicial operad which we also denote by $\P$. Clearly, we have an equivalence of $\infty$-operads $\N_\cS\P\simeq\Ass$.

Let $\cM$ be a monoidal model category, Quillen equivalent to a combinatorial model category. Then $\cM$ is an algebra over $\P$ in $\MC^c$ (as operads in $\Cat_1$ and thus in $\sS$). Applying the simplicial nerve we get that $\cM$ is an algebra over $\N_\cS\P\simeq \Ass$ in $\N_\cS(\MC^c)$. It follows that $\cM_\infty$ is an algebra over $\Ass$ in $\Catc$, so  $\cM_\infty$ is a presentable closed monoidal $\infty$-category. Furthermore, the localization functor
$$\cM\to\cM_\infty$$
is lax monoidal.

Now let $\cC$ be a model category, Quillen equivalent to a combinatorial one. Suppose that $\cC$ is an $\cM$-model category, in the sense that $\cC$ is tensored closed over $\cM$ and satisfies the Quillen SM7 axiom.
As above, we can construct a simplicial operad $\Q$ whose simplicial nerve is equivalent to $\LM$ and such that $(\cM,\cC)$ is an algebra over $\Q$ in $\MC^c$.
Applying the simplicial nerve we get that $(\cM,\cC)$ is an algebra over $\N_\cS\Q\simeq \LM$ in $\N_\cS(\MC^c)$.
It follows that $(\cM_\infty,\cC_\infty)$ is an algebra over $\LM$ in $\Catc$, so $\cC_\infty$ is a presentable $\infty$-category left tensored over $\cM_\infty$. Furthermore, the localization functor
$$(\cM,\cC)\to(\cM_\infty,\cC_\infty)$$
is $\LM$-lax monoidal.

Let $(-)^f$ and $(-)^c$ denote fibrant and cofibrant replacement functors in a model category. The model category $\cC$ is an $\cM$-model category so for every $A\in\cC$ we have a Quillen pair
$$(-)\otimes A^c:\cM\adj\cC:\Hom_\cC(A^c,-).$$
By \cite{Maz} we have an induced adjunction of $\infty$-categories
$$\mathbb{L}(-)\otimes A^c:\cM_\infty\adj\cC_\infty:\RR \Hom_\cC(A^c,-).$$
Thus we have equivalences natural in $A,B\in\cC$:
$$\Map_{\cM_\infty}(K,\RR \Hom_\cC(A^c,B))\simeq \Map_{\cC_\infty}(\mathbb{L}K\otimes A^c,B)$$
Clearly $\mathbb{L}(-)\otimes A^c:\cM_\infty\to\cC_\infty$ represents the tensor product $(-)\otimes A$ of $\cC_\infty$ as tensored over $\cM_\infty$ so we have natural equivalences
$$\Map_{\cM_\infty}(K,\Hom_\cC(A^c,B^f))\simeq \Map_{\cC_\infty}(K\otimes A,B).$$
We see that $\Hom_\cC(A^c,B^f)\in \cM_\infty$ is a representing object for $\Hom_{\cC_\infty}(A,B)\in P(\cM_\infty)$ and thus we have
$$\Hom_\cC(A^c,B^f)\simeq \Hom_{\cC_\infty}(A,B).$$

\begin{define}\label{d:enriched localization}
Let $\cA$ be a strictly enriched category over $\cM$. Let $X$ denote the discrete space on the set of objects of $\cA$. Then $\cA$ is an algebra over $\Ass$ in the monoidal category $\Quiv_X(\cM)$ (see \cite{Hin4}).
The localization functor
$$\cM\to\cM_\infty$$
is lax monoidal, so the induced functor
$$\Quiv_X(\cM)\to\Quiv_X(\cM_\infty)$$
is also lax monoidal. We define the \emph{enriched $\infty$-localization} functor to be composition with this last functor:
$$(-)_\infty:\Alg_{\Ass}(\Quiv_X(\cM))\to \Alg_{\Ass}(\Quiv_X(\cM_\infty)).$$
Thus, $\cA_\infty$ is an enriched $\infty$-category over $\cM_\infty$.

Let $F:\cA\to\cC$ be a strict $\cM$-functor.
We have an $\LM$-monoidal category
$$\Quiv^{\LM}_X(\cM,\cC):=(\Quiv_X(\cM),\Fun(X,\cC))$$
and $F$ is just a module in $\Fun(X,\cC)$ over $\cA$ (see \cite{Hin4}). In this case $(\cA,F)$ is an algebra over $\LM$ in the $\LM$-monoidal category $\Quiv^{\LM}_X(\cM,\cC)$.
The localization functor
$$(\cM,\cC)\to(\cM_\infty,\cC_\infty)$$
is $\LM$-lax monoidal, so the induced functor
$$\Quiv^{\LM}_X(\cM,\cC)\to\Quiv^{\LM}_X(\cM_\infty,\cC_\infty)$$
is also $\LM$-lax monoidal. It follows that we obtain a functor that we denote
$$(-)_\infty:\Alg_{\LM}(\Quiv^{\LM}_X(\cM,\cC))\to \Alg_{\LM}(\Quiv^{\LM}_X(\cM_\infty,\cC_\infty)).$$
Clearly this functor lifts the functor above so we have
$$(\cA,\cF)_\infty=(\cA_\infty,F_\infty).$$
Thus, $F_\infty$ is an $\cM_\infty$-functor from $\cA_\infty$ to $\cC_\infty$ and we call it the \emph{enriched $\infty$-localization} of $F$.

We call $F$ \emph{homotopy fully faithful} if for every $x,y\in X$ the composition
$$\cA(x,y)\xrightarrow{F} \Hom_\cC(F(x),F(y))\to \Hom_\cC(F(x)^c,F(y)^f)$$
is an equivalences in the model category $\cM$.
\end{define}

\begin{thm}
  Let $\cA$ be a strictly enriched category over $\cM$ and let $F:\cA\to\cC$ be a strict $\cM$-functor which is homotopy fully faithful. Then the $\cM_\infty$-functor $F_\infty:\cA_\infty\to \cC_\infty$ is $\cM_\infty$-fully faithful (in the sense described after Remark \ref{r:enriched model}).
\end{thm}

\begin{proof}
Lat $L:\cM\to\cM_\infty$ denote the localization functor. Then for every $x,y\in X$ we have a commutative square in $\cM_\infty$
$$\xymatrix{L(\cA(x,y))\ar[drr]\ar[rr]^{LF}\ar[d] & & L(\Hom_{\cC}(F(x),F(y)))\ar[d]\\
\cA_\infty(x,y)\ar[rr]^{F_\infty} & & \Hom_{\cC_\infty}(F_\infty(x),F_\infty(y)).}$$
The left map is an equivalence by definition of $\cA_\infty$ and the right map is equivalent to $L$ applied to the map $\Hom_\cC(F(x),F(y))\to \Hom_\cC(F(x)^c,F(y)^f)$.
Since $F$ is homotopy fully faithful, the diagonal map is an equivalence, and thus also the bottom map.
\end{proof}

\begin{cor}\label{c:fully faithful}
  Let $\cA$ be a strictly enriched category over $\cM$ and let $F:\cA\to\cC$ be a strict fully faithful $\cM$-functor that lands in the fibrant cofibrant objects. Then the $\cM_\infty$-functor $F_\infty:\cA_\infty\to \cC_\infty$ is $\cM_\infty$-fully faithful.
\end{cor}

\subsubsection*{Enriched Yoneda Lemma} Hinich formulates and proves a version of enriched Yoneda lemma, which is of key importance to us. We will review this part of Hinich's work next.

Let $\cM$ be a closed monoidal $\infty$-category and let $\cA$ be an $\cM$-enriched category with space of objects $X$. Hinich defines the opposite category $\cA^{\op}$, which is an $\cM^{rev}$-enriched category with space of objects $X^{\op}$, and constructs a structure of a category left-tensored over $\cM$ on the $\infty$-category of $\cM$-presheaves
$$P_{\cM}(\cA):=\Fun_{\cM^{rev}}(\cA^{\op},\cM).$$
Here, $\cM$ is considered as a right $\cM$-module which is the same as a left $\cM^{rev}$-module.
\begin{rem}
In the case of interest to us, $\cM$ is the category of spectra, which is a {\it symmetric} monoidal category. This means that there is a canonical equivalence of monoidal categories $\cM\simeq \cM^{rev}$.
\end{rem}
Hinich also constructs an $\cM$-fully faithful functor called the enriched Yoneda embedding
$$Y:\cA\to P_{\cM}(\cA).$$
In \cite{Hin3} it is shown that this construction has the following universal property: If $\cB$ is any category left-tensored over $\cM$ then precomposition with $Y$ induces an equivalence
$$\Map_{\LMod_{\cM}}(P_{\cM}(\cA),\cB)\simeq\Map_{\cM}(\cA,\cB).$$

In \cite{Hin3}, all the above is done more generally relative to an $\infty$-operad $\cO$. Taking $\cO=\Com$ to be the terminal $\infty$-operad and noting that $\Com\otimes\Ass\simeq \Com$, we obtain the following. Suppose $\cM$ is a closed symmetric monoidal $\infty$-category. Then the category of $\cM$-left-tensored categories is symmetric monoidal and we define a symmetric monoidal $\cM$-left-tensored category to be a commutative algebra in the category of $\cM$-left-tensored categories. Similarly, the category of $\cM$-enriched categories is symmetric monoidal and we define a symmetric monoidal $\cM$-enriched category to be to be a commutative algebra in the category of $\cM$-enriched categories. Moreover, one can define the notion of a symmetric monoidal $\cM$-functor from a symmetric monoidal $\cM$-enriched category to a symmetric monoidal $\cM$-left-tensored category.

If $\cA$ is a symmetric monoidal $\cM$-enriched category, the category of presheaves $P_\cM (\cA)$ acquires a canonical symmetric monoidal $\cM$-left-tensored structure (Day convolution), and the Yoneda embedding $Y : \cA \to P_\cM (\cA)$ acquires a structure of a symmetric monoidal $\cM$-functor. Moreover, this construction has the following universal property: If $\cB$ is any symmetric monoidal $\cM$-left-tensored category then precomposition with $Y$ induces an equivalence
$$\Map^\Com_{\LMod_{\cM}}(P_{\cM}(\cA),\cB)\simeq \Map^\Com_{\cM}(\cA,\cB).$$

\section{Stable $\infty$-categories and spectral presheaves}\label{s:stab_infinity}
In this section we consider the notion of stable $\infty$-categories. We show that a compactly generated stable $\infty$-category is equivalent to $P_\Sp(\cA)$ for some small $\Sp$-enriched category $\cA$.

Let $\cD$ be a pointed finitely cocomplete $\infty$-category. We define the \emph{suspension functor} on $\cD$
$$\Sigma_\cD\colon \cD\to \cD$$
by the formula
$$\Sigma_\cD(X):=*\coprod_X *.$$
Alternatively, the suspension functor can be defined as the smash product $S^1\wedge X$, using the fact that a pointed finitely cocomplete $\infty$-category is tensored over pointed spaces.

If the suspension functor is an equivalence of categories, then $\cD$ is called \emph{stable}. 
 A stable presentable $\infty$-category is naturally left-tensored over the closed monoidal $\infty$-category of spectra $\Sp$ (see \cite[Proposition 4.8.2.18]{Lur2}). Moreover, $\Sp$ is presentable, so for every $b,c\in\cD$ the presheaf $\Hom_\cD(b,c)\in P(\Sp)$ is representable (we will denote the representing object also by $\Hom_\cD(b,c)\in\Sp$). By Proposition~\ref{proposition: enrichment} it follows that a stable presentable $\infty$-category is canonically enriched over $\Sp$ (this was observed by Gepner-Haugseng in~\cite[Example 7.4.14]{GH}, where they also pointed out that the presentability assumption is unnecessary).


\begin{thm}\label{t:modules}
  Let $\cD$ be a cocomplete stable $\infty$-category. Suppose that there is a small set $C$ of compact objects in $\cD$, that generates $\cD$ under colimits and desuspensions. Thinking of $\cD$ as left-tensored over $\Sp$, we let $\cC$ be the full $\Sp$-enriched subcategory of $\cD$ spanned by $C$. Then we have a natural functor of categories left-tensored over $\Sp$
  $$P_{\Sp}(\cC)\xrightarrow{\sim}\cD,$$
  which is an equivalence of the underlying $\infty$-categories and sends each representable presheaf $Y(c) \in P_{\Sp}(\cC)$ to $c\in C$.
\end{thm}

\begin{rem}
  Theorem \ref{t:modules} appears in \cite[Theorem 7.1.2.1]{Lur2} for the case that $|C|=1$. The general case, formulated in the language of model categories, can be found in \cite[Theorem 3.3.3]{ScSh1}. In \cite{GM} the last result can be found under more general hypotheses
\end{rem}

\begin{proof}
By defintion of full enriched subcategory, there is a fully-faithful $\Sp$-functor $i:\cC\to\cD$. By the universal property of the Yoneda embedding, we have an induced functor of categories left-tensored over $\Sp$
$$I\colon P_{\Sp}(\cC)\to\cD,$$
such that $I\circ Y\simeq i$. We thus get an equivalence $I(Y(c))\simeq i(c)\simeq c$, for every $c\in\cC$, and it remains to show that $I$ is an equivalence.

The functor $I$ is a morphism of left modules over $\Sp$ in $\Catc$, so in particular $I$ commutes with colimits.
The $\infty$-category $\cD$ is presentable by \cite[Theorem 5.5.1.1]{Lur1}. Let $X$ denote the full subspace of $\cD$ generated by $C$ and recall that the $\infty$-category $P_{\Sp}(\cC)$ is defined as the category of left $\cC^{\op}$-modules with
values in $\Fun(X^{\op},\Sp)$.
Since $\Fun(X^{\op},\Sp)$ is stable and presentable, so is $P_{\Sp} (\cC)$ (see \cite[1.1.3.1 and 4.2.3.5]{Lur2}).
Thus, by the adjoint functor theorem $I$ has a right adjoint $J$:
$$I\colon P_{\Sp}(\cC)\adj\cD \noloc J.$$

We first show that the unit $Y(c)\to J(I(Y(c)))$ of the adjunction $I\dashv J$ is an equivalence for every $c\in\cC$. It is not hard to show that $J$ preserves $\Sp$-enrichment, so both $Y$ and $J\circ I\circ Y$ are $\Sp$-functors $\cC\to P_{\Sp}(\cC)$, and that the unit induces a map $Y\to J\circ I\circ Y$ of $\Sp$-functors.
Note that
$$\Fun_{\Sp}(\cC,P_{\Sp}(\cC))= \LMod_{\cC}(\Fun(X,P_{\Sp}(\cC)))=$$
$$\LMod_{\cC}(\Fun(X,\Fun_{\Sp^{rev}}(\cC^{\op},\Sp)))=\LMod_{\cC}(\Fun(X,\LMod_{\cC^{\op}}(\Fun(X^{\op},\Sp))))= $$
$$\LMod_{\cC}(\LMod_{\cC^{\op}}(\Fun(X,\Fun(X^{\op},\Sp))))=
\LMod_{\cC}(\RMod_{\cC}(\Fun(X^{\op}\times X,\Sp))),$$
so an $\Sp$-functor $\cC\to P_{\Sp}(\cC)$ is the same as a $\cC$-$\cC$-bimodule in the category $\Fun(X^{\op}\times X,\Sp)$. Thus we can view $Y\to J\circ I\circ Y$ as a map of $\cC$-$\cC$-bimodules in $\Fun(X^{\op}\times X,\Sp)$ and we need to show that it is an equivalence. The forgetful functor to $\Fun(X^{\op}\times X,\Sp)$ reflects equivalences,
and an equivalence in $\Fun(X^{\op}\times X,\Sp)$ can be verified objectwise, so we can fix two objects $c,d$ in $\cC$ and show that the induced map of spectra
$$Y(c,d)\to (J\circ I\circ Y)(c,d)$$
is an equivalence. But since $Y$ and $i$ are $\Sp$-fully faithful, we have
$$(J\circ I\circ Y)(c,d)=J(I(Y(d)))(c)\simeq$$
$$\Hom_{P_{\Sp} (\cC)} (Y (c),J(I(Y(d)))) \simeq \Hom_{\cD} (I(Y(c)),I(Y(d)))\simeq$$
$$ \Hom_{\cD}(i(c),i(d))\simeq{\cC}(c,d)\simeq\Hom_{P_{\Sp} (\cC)} (Y (c),Y(d)) \simeq Y(d)(c)\simeq Y(c,d).$$

Since $I(Y(c))\simeq c$ and $J(c)\simeq Y(c)$, the counit $I(J(c))\to c$ of $I\dashv J$ is also an equivalences, for every $c\in C$. Note that $C$ generates $\cD$ under colimits, $\{Y(c)|c\in C\}$ generates $P_{\Sp}(\cC)$ under colimits and the functor $I$ commutes with colimits. Thus, if we can show that $J$ also commutes with colimits, it would follow that the unit and counit of $I\dashv J$ are equivalences, and we are done.

Let us show first that $J$ commutes with filtered colimits. So let $d=\colim_{i\in I} d_i$ be a filtered colimit diagram in $\cD$. We need to verify that the induced map $\colim_{i\in I} J(d_i)\to J(d)$ is an equivalence.
Recall that
$$P_{\Sp}(\cC)=\Fun_{\Sp^{rev}}(\cC^{\op},\Sp)= \LMod_{\cC^{\op}}(\Fun(X^{\op},\Sp)).$$
The forgetful functor $U$ from $P_{\Sp}(\cC)$ to $\Fun(X^{\op},\Sp)$
commutes with colimits (see \cite[4.2.3.5]{Lur2}) and reflects equivalences, so it is enough to verify that
$${\colim}_{i\in I} U(J(d_i))\to U(J(d))$$
is an equivalence.
Now colimits in $\Fun(X^{\op},\Sp)$ are pointwise, so we can fix $c \in \cC$ and show that
$${\colim}_{i\in I} U(J(d_i))(c)\to U(J(d))(c)$$
in an equivalence.
We have an equivalence natural in $e\in \cD$
$$U(J(e))(c)=J(e)(c)=\Hom_{P_{\Sp} (\cC)} (Y (c),J(e))\simeq$$ $$
\Hom_{\cD}(I(Y(c)),e)\simeq\Hom_{\cD}(i(c),e)\simeq \Hom_\cD(c,e),$$
so it is enough to show that
$${\colim}_{i\in I} \Hom_\cD(c,d_i)\to\Hom_\cD(c,d)$$
is an equivalence, which is true by the compactness of $c$ in $\cD$.

Since both range and domain of $J$ are stable, and in a stable $\infty$-category every pullback square is a pushout square and vice versa, it follows that $J$ sends pushout squares to pushout squares. Thus $J$ commutes with all small colimits.
\end{proof}

Using the results in \cite{Hin3} one can prove an extension to Theorem \ref{t:modules}:

\begin{thm}\label{t:modules monoidal}
In the situation of Theorem \ref{t:modules}, suppose $\cD$ is symmetric monoidal and the set $C$ is closed under the monoidal product in $\cD$ and contains the unit of $\cD$. Then $\cC$ acquires a canonical symmetric monoidal $\Sp$-enriched structure, the category of presheaves $P_\Sp (\cC)$ acquires a canonical symmetric monoidal left $\Sp$-tensored structure and the equivalence $P_{\Sp}(\cC)\xrightarrow{\sim}\cD$ acquires a canonical symmetric monoidal left $\Sp$-tensored structure.
\end{thm}


\section{Stabilization of categories}\label{ss:stab}

In this section we review the notion of stabilization of an $\infty$-category. We will use the framework established by Lurie in~\cite[Section 1.4]{Lur2}. We present in subsection \ref{ss: spectral topological} a similar procedure that can be applied to an ordinary topologically enriched category. We will compare the strict and the $\infty$-categorical versions of stabilization.

Let $\cC$ be a pointed $\infty$-category. To ensure that $\cC$ has all the good properties we may want, we will assume that $\cC\simeq \Ind(\cC_0)$ where $\cC_0$ is a small pointed $\infty$-category that is closed under finite colimits. This includes $\NCW\simeq \Ind(\FNCW)$. By~\cite[Theorem 5.5.1.1]{Lur1} $\cC$ is presentable, and therefore has small limits and colimits [op. cit., Corollary 5.5.2.4]. Furthermore, filtered colimits commute with finite limits in $\cC$ by the remark immediately following \cite[Definition 5.5.7.1]{Lur1}. It follows, in particular, that $\cC$ is {\it differentiable} in the sense of~\cite[Definition 6.1.1.6]{Lur2} and therefore the results of [op. cit., Chapter 6] apply to $\cC$.

Recall that $\FCW$ is the $\infty$-category of pointed finite CW-complexes. Let $F\colon\FCW\to \cC$ be a functor. Recall that $F$ is called {\it reduced} if $F(*)$ is a final object of $\cC$, and $F$ is called {\it $1$-excisive} if $F$ takes pushout squares to pullback squares. Let {\it linear} functors be functors that are both reduced and $1$-excisive. Linear functors provide a good framework for defining spectra in the context of general $\infty$-categories.
\begin{define}[\cite{Lur2}, Definition 1.4.2.8]\label{d:SpC}
A {\it spectrum object} in $\cC$ is a linear functor $\FCW\to \cC$. Let $\Sp(\cC)$ be the $\infty$-category of linear functors $\FCW\to \cC$.
\end{define}
$\Sp(\cC)$ is called the category of spectra of $\cC$, or the stabilization of $\cC$. By results in~\cite{Lur2}, $\Sp(\cC)$ is a stable and presentable $\infty$-category (Corollary 1.4.2.17 and Proposition 1.4.4.4 respectively).

Let $\Fun_*(\FCW, \cC)$ be the category of all pointed functors from $\FCW$ to $\cC$. Then $\Sp(\cC)$ is by definition a full subcategory of $\Fun_*(\FCW, \cC)$. The fully faithful functor $\Sp(\cC)\hookrightarrow \Fun_*(\FCW, \cC)$ has a left adjoint $$\Lin\colon \Fun_*(\FCW, \cC) \to \Sp(\cC)$$ called linearization. Explicitly, if $F\colon \FCW \to \cC$ is a pointed functor, then the linearization of $F$ is given by the following formula
\[
\Lin F(X)=\colim_{n\to \infty}\Omega^n_{\cC} F(\Sigma^n X).
\]
See~\cite[Example 6.1.1.28]{Lur2} for a discussion of this formula in the context of $\infty$-categories (of course this formula is older than~\cite{Lur2} and goes back at least to~\cite{Goo}).
The stabilization $\Sp(\cC)$ is the left Bousfield localization of $\Fun_*(\FCW, \cC)$ at the stable equivalences and $\Lin$ is the localization functor (a map between functors is a stable equivalence if it induces an equivalence between linearizations).

There is an adjoint pair of functors
\[
\Sigma_{\cC}^{\infty} \colon \cC \leftrightarrows  \Sp(\cC) \noloc \Omega^{\infty}_{\cC},
\]
where $\Sigma^\infty_{\cC} x(K)=\colim_{n\to \infty} \Omega^n_{\cC} \Sigma^n_{\cC} (K \wedge x)$ and $\Omega^{\infty}_{\cC} G=G(S^0)$. This formula for $\Omega^\infty_{\cC}$ agrees with the one in~\cite[Notation 1.4.2.20]{Lur2}, and therefore our $\Sigma^\infty_{\cC}$, being left adjoint to $\Omega^\infty_{\cC}$ is also equivalent to Lurie's. The functor $\Sigma^\infty_{\cC}$ satisfies the following universal property: For every stable presentable $\infty$-category $\cD$, pre-composition with $\Sigma^{\infty}_\cC$ induces an equivalence of $\infty$-categories
$$\Funl(\Sp(\cC),\cD)\xrightarrow{\simeq} \Funl(\cC,\cD),$$
where $\Funl$ denotes left functors (that is, colimit preserving functors).

An important special case is when $\cC=\cS_*$ is the $\infty$-category of pointed spaces. In this case $\Sp:=\Sp(\cS_*)$ is the classical $\infty$-category of spectra, presented as the category of linear functors from $\FCW$ to $\cS_*$.
Whenever $\cC=\cS_*$ we write $\Sp$, $\Sigma^\infty$ or $\Omega^\infty$, omitting the subscript $\cC$.

There is another useful way to construct $\Sp(\cC)$ when $\cC=\Ind(\cC_0)$, with $\cC_0$ closed under finite colimits. We will now describe it.
\begin{define}
Let $\cC_0$ be an $\infty$-category closed under finite colimits. Define the Spanier-Whitehead category of $\cC_0$,
which we denote by $\SW(\cC_0)$, to be the colimit of the sequence
$$\cC_0\xrightarrow{\Sigma_{\cC_0}}\cC_0\xrightarrow{\Sigma_{\cC_0}}\cdots$$
in $\Cat$.
\end{define}
Thus, the objects
of $\SW(\cC_0)$ are pairs $(X,n)$ where $X\in\cC_0$ and $n\in \mathbb{N}$. The pair $(X, n)$ represents the $n$-fold desuspension of $X$. The mapping spaces in $\SW(\cC_0)$
are given by
$$\Map_{\SW(\cC_0)}((X,n),(Y,m)) = {\colim}_{k\in\NN}\Map_{\cC_0}(\Sigma_\cC^{k-n} X,\Sigma_\cC^{k-m} Y),$$
where the colimit is taken in the $\infty$-category of spaces. Clearly, $\SW(\cC_0)$ is a stable $\infty$-category. It is closed under finite colimits, but not under arbitrary colimits. It plays the role of the category of finite spectra over $\cC$.
There is a finite suspension spectrum functor
$${\Sigma^{\infty}_{\cC_0}}^f:\cC_0\to \SW(\cC_0)$$
given by $X\mapsto (X,0)$, which satisfies the following universal property: For every stable $\infty$-category $\cD$, pre-composition with ${\Sigma^{\infty}_{\cC_0}}^f$
induces an equivalence of $\infty$-categories
$$\Funfinc(\SW(\cC_0),\cD)\xrightarrow{\simeq}\Funfinc(\cC_0,\cD),$$
where $\Funfinc$ denotes pointed finite colimit preserving functors.

One has the following description of $\Sp(\Ind(\cC_0))$:
\begin{prop}\label{p:stab ind}
Let $\cC_0$ be a small pointed finitely cocomplete ${\infty}$-category, and let $\cC:=\Ind(\cC_0 )$. Then there is a natural equivalence
$$\Sp(\cC)\simeq\Ind(\SW(\cC_0 )).$$
and under this equivalence,
$$\Sigma^{\infty}_\cC:\Ind(\cC_0 )\simeq\cC\to \Sp(\cC)\simeq \Ind(\SW(\cC_0 )),$$
is just the prolongation of
$${\Sigma^{\infty}_{\cC_0}}^f:{\cC_0}\to \SW({\cC_0}).$$
\end{prop}
\begin{proof}
This is proved in~\cite[Chapter 1.4]{Lur2} in the case when $\cC_0=\FCW$ and $\cC=\Ind(\FCW)\simeq \cS_*$. The proof in the general case is similar. In brief, one can check that $\Ind(\SW(\cC_0 ))$ satisfies the same universal property as $\Sp(\cC)$.
\end{proof}
This proposition has the following rather important corollary.
\begin{cor}\label{cor: generates}
Let $\cC_0, \cC$ be as before. Suppose $A$ is a set of objects of $\cC_0$ that generates $\cC_0$ under finite colimits, in the sense that $\cC_0$ is the only subcategory of $\cC_0$ that contains $A$ and is closed under finite colimits and equivalences. Then the set of suspension spectra $\{\Sigma^\infty_{\cC} x\mid x\in A\}$ generates $\SW(\cC_0)$ under finite colimits and desuspensions, and generates $\Sp(\cC)$ under arbitrary colimits and desuspensions.
\end{cor}

The following notation is going to be used quite a lot.
\begin{define}\label{d:G infinity}
Suppose $\cC$ is a pointed $\infty$-category with finite colimits. Let $x$ and $y$ be object of $\cC$. Then the functor $G^{\cC}_{x,y}\colon \FCW\to \cS_*$ is defined by $G^{\cC}_{x,y}(K)=\Map_{\cC}(x, K\wedge y)$. Sometimes we will omit the superscript $\cC$ and write simply $G_{x, y}$.
\end{define}

\begin{lem}
If $\cC$ is a stable $\infty$-category then $G_{x, y}$ is linear for any two objects $x, y$.
\end{lem}
\begin{proof}
We need to prove that $G_{x, y}(*)\simeq *$ and that $G_{x, y}$ is $1$-excisive. The first condition holds because $*\wedge y$ is equivalent to a final object of $\cC$. Now let us prove that $G_{x, y}$ is $1$-excisive.
We have equivalences
\[
\Map_{\cC}(x, K\wedge y) \xrightarrow{\simeq} \Map_{\cC}(S^1\wedge x, S^1 \wedge K \wedge y) \xrightarrow{\simeq} \Map_{\cC}(x, \Omega (S^1\wedge K \wedge y)).
\]
Here the first map is an equivalence because $\cC$ is stable, and the second equivalence is a standard adjunction. The composite equivalence can be reinterpreted as saying that the canonical map $G_{x, y}\to \Omega G_{x, y}\Sigma$ is an equivalence. It follows that the map $G_{x, y}\to \Lin G_{x, y}$ is an equivalence, so $G_{x, y}$ is linear.
\end{proof}
If $\cC$ is a stable $\infty$-category then $G^{\cC}_{x,y}$ is in fact equivalent to the canonical enrichment of $\cC$ over spectra. The next lemma and remark say that more generally the linearization of $G^{\cC}$ gives, in favorable circumstances, the spectral enrichment of the stabilization of $\cC$.
\begin{lem}\label{lemma: faithful}
Let $\cC_0$ be a small finitely cocomplete $\infty$-category. To simplify notation, let $S\colon \cC_0\to \SW(\cC_0)$ be the finite suspension spectrum functor. Then the natural map
\[
G^{\cC_0}_{x, y}\to G^{\SW(\cC_0)}_{S(x), S(y)}
\]
induced by the finite suspension spectrum functor is equivalent to the linearization map
\[
G^{\cC_0}_{x, y}\to  \Lin G^{\cC_0}_{x, y}.
\]
\end{lem}
\begin{proof}
By definition, we have equivalences
\begin{multline*}
G^{\SW(\cC_0)}_{S(x), S(y)}(K)=\Map_{\SW(\cC_0)}(S(x), K\wedge S(y))= \colim_{n\to \infty} \Map_{\cC_0}(S^n \wedge x, S^n \wedge K\wedge y)=\\
 = \colim_{n\to \infty} \Omega^n \Map_{\cC_0}(x, S^n \wedge K\wedge y)= \colim_{n\to \infty} \Omega^n G^{\cC_0}_{x, y} (S^n \wedge K)= \Lin G^{\cC_0}_{x, y}(K)
\end{multline*}
and the map from $G^{\cC_0}_{x, y}$ is precisely the linearization map.
%
%
%
\end{proof}
\begin{rem}
If $\cC=\Ind \cC_0$, with $\cC_0$ as in the previous lemma, and $x, y$ are objects of $\cC$, it is not true that $G^{\Sp(\cC)}_{\Sigma^\infty_{\cC}(x) ,\Sigma^\infty_{\cC}(y)}$ is equivalent to the linearization of $G^{\cC}_{x, y}$. Rather, there is an equivalence
\[
G^{\Sp(\cC)}_{\Sigma^\infty_{\cC}(x) ,\Sigma^\infty_{\cC}(y)}(K)= \Map_{\cC}(x, \colim_{n\to \infty} \Omega^n (S^n \wedge y)).
\]
This is equivalent to $\Lin G^{\cC}_{x, y}(K)$ if $x$ is a compact object, but not in general. In particular, it is true when $x\in \cC_0$, which is the case considered in the previous lemma.
\end{rem}

\subsubsection{Spectral enrichment of pointed topological categories}\label{ss: spectral topological}
Suppose $\cC$ is an $\infty$-category and $x, y$ are objects of $\cC$. We saw that functors of the form $G^{\cC}_{x, y}$ can be used to define the spectral enrichment of the stabilization of $\cC$. If $\cC$ is an ordinary topological category, one can use similar functors $G^{\cC}_{x, y}$ to define a spectral enrichment of $\cC$, using the more traditional view of spectra as modeled by the Quillen model category of continuous functors. In this subsection we define a strict spectral enrichment of pointed topological categories and compare it with the $\infty$-categorical enrichment.

Let $\Top$ denote the category of pointed compactly generated weak Hausdorff spaces with the standard model structure of Quillen \cite{Qui}. Every object in $\Top$ is fibrant, and every CW-complex is cofibrant. The model category $\Top$ is a model for the $\infty$-category of pointed spaces $\cS_*$. This means that the $\infty$-localization of $\Top$ with respect to the weak equivalences (see Remark \ref{r:coherent}) is canonically equivalent to $\cS_*$.
Note that any pointed topological category is naturally enriched in $\Top$.



\begin{define}\label{d:CW_f tensored}
A pointed topological category $\cC$ is called tensored closed over $\FCW$ if we are given a bi-continuous left action $\wedge:\FCW \times \cC\to \cC$ such that the following hold:
\begin{enumerate}
  \item The $\infty$-category $\cC_\infty$ is finitely cocomplete, where $\cC_\infty$ is the topological nerve of $\cC$.
  \item After application of the topological nerve the functor
  $$\wedge_\infty:\FCW \times \cC_\infty\to \cC_\infty$$
  commutes with finite colimits in each variable.
\end{enumerate}
\end{define}

\begin{define}\label{def: G}
Let $\cC$ be a pointed topological category, tensored closed over $\FCW$.
Let $x$ and $y$ be objects of $\cC$. In keeping with notation we introduced in Definition~\ref{d:G infinity}, we define the pointed topological functor $G^{\cC}_{x,y}\colon \FCW\to \Top$ by $G^{\cC}_{x,y}(K)=\Map_{\cC}(x, K\wedge y)$. Again, we may omit the superscript $\cC$ and write simply $G_{x, y}$.
\end{define}

\begin{rem}\label{r:G}
We saw earlier that pointed $\infty$-functors from $\FCW$ to $\cS_*$ provide a way of defining the $\infty$-category of spectra. This is known also in the more traditional approach to spectra via model categories. There is a Quillen model structure on the category $\Fun_*(\FCW, \Top)$ of continuous pointed functors, called the stable model structure, and it provides one of the models for the category of spectra.
We refer the reader to \cite{Lyd, MMSS} for more details about this model structure.
\end{rem}

We denote the category $\Fun_*(\FCW, \Top)$ with the stable model structure by $\SpM$.
We denote the category $\Fun_*(\FCW, \Top)$ with the \textbf{projective} model structure simply by $\Fun_*(\FCW, \Top)$.
The model category $\SpM$ is a left Bousfield localization of $\Fun_*(\FCW, \Top)$, so we have a Quillen pair
$$\id:\Fun_*(\FCW, \Top)\rightleftarrows\SpM:\id.$$
Applying $\infty$-localization, this Quillen pair becomes the localization adjunction
$$\cL:\Fun_*(\FCW, \cS_*)\rightleftarrows\Sp:\i.$$

Let $\cC$ be a pointed topological category, tensored closed over $\FCW$ and let $x,y\in\cC$. By the previous definition we have a functor $G^{\cC}_{x,y}\colon \FCW\to \Top$. Under the identification $\Fun_*(\FCW, \Top)_\infty\simeq \Fun_*(\FCW, \cS_*)$, $G^{\cC}_{x,y}$ can be thought of as an object in $\Fun_*(\FCW, \cS_*)$. We will now compare this with the functor
$G^{\cC_\infty}_{x,y}\colon \FCW\to \cS_*$ from Definition \ref{d:G infinity}.

Clearly, the bi-functor
$$\wedge_\infty:\FCW \times \cC_\infty\to \cC_\infty$$
is a left action of the monoidal $\infty$-category $\FCW$ on $\cC_\infty$. It follows that we have an induced action on the ind-categories
$$\wedge_\infty:\cS_* \times \Ind(\cC_\infty)\to \Ind(\cC_\infty)$$
and this action commutes with small colimits in each variable.
Since $\cS_*$ is a mode in the sense of \cite[Section 5]{CSY}, this action coincides with the canonical action of $\cS_*$ on $\Ind(\cC_\infty)$ as a presentable pointed $\infty$-category.
In particular we get that the restriction
$$\wedge_\infty:\FCW \times \cC_\infty\to \cC_\infty$$
coincides with the canonical action of $\FCW$ on $\cC_\infty$ as an $\infty$-category with finite colimits.

It follows that under the identification $\Top_\infty\simeq \cS_*$, for any $K\in\FCW$ and $z\in\cC$ we have natural equivalences
  $$\Map_{\cC}(x, z)\simeq\Map_{\cC_\infty}(x, z)$$
  $$K\wedge y\simeq K\wedge_{\infty} y.$$
  Thus we have
\begin{equation}\label{eq:G}
G^{\cC_\infty}_{x,y}(K)=\Map_{\cC_\infty}(x, K\wedge_{\infty} y)\simeq \Map_{\cC}(x, K\wedge y)=G^{\cC}_{x,y}(K)\end{equation}
   or $G^{\cC_\infty}_{x,y}\simeq G^{\cC}_{x,y}.$


\begin{define}\label{d:SpM enrichment}
  Let $\cC$ be a pointed topological category, tensored closed over $\FCW$. We define a strict enrichment of $\cC$ over $\SpM$ as follows: If $x$ and $y$ are objects of $\cC$ we define $$\Hom_\cC(x,y):=G^\cC_{x, y}\in\SpM.$$
  Let $x,y,z$ be objects of $\cC$ and let $K$ and $L$ be finite CW-complexes. Note that there is a natural map $G_{x,y}(K)\wedge G_{y,z}(L)\to G_{x, z}(K\wedge L)$, defined as the composite:
\begin{multline*}
\Map_{\cC}(x, K\wedge y)\wedge \Map_{\cC}(y, L\wedge z)\to \\ \to \Map_{\cC}(x, K\wedge y)\wedge \Map_{\cC}(K\wedge y, K\wedge L\wedge z)\to \\ \to\Map_{\cC}(x, K\wedge L\wedge z)
\end{multline*}
where the first map is induced by the topological functor $K\wedge (-):\cC\to\cC$ and the second map is given by composition.
This map induces a natural map $G_{x, y}\otimes G_{y,z}\to G_{x, z}$, where $\otimes$ denotes Day convolution which is the tensor product in $\SpM$. Thus we have defined composition and it can be checked that the above indeed defines a strict enrichment of $\cC$ over $\SpM$.
\end{define}


\begin{thm}\label{t:mapping}
Let $\cC$ be a small pointed topological category, tensored closed over $\FCW$.
Then under the identification
$\SpM_\infty\simeq \Sp$,
for any two objects $x$ and $y$ of $\cC$ we have a natural equivalence
 $$\Hom_\cC(x,y)\simeq \Hom_{\Sp(\Ind(\cC_\infty))}(\Sigma^\infty(x),\Sigma^\infty(y)).$$
\end{thm}

\begin{proof}
Let $x$ and $y$ be objects in $\cC$. Recall that $\Hom_\cC(x,y)=G^\cC_{x,y}$ and consider $G^{\cC}_{x,y}$ as an object in $\Fun_*(\FCW, \Top)$. By \eqref{eq:G}, we have $G^{\cC}_{x,y}\simeq G^{\cC_{\infty}}_{x,y}.$

We have a commutative square
$$\xymatrix{\Fun_*(\FCW, \Top)_\infty\ar[d]_{\mathbb{L}\id}\ar[r]^{\sim} & \Fun_*(\FCW, \cS_*)\ar[d]^{\cL}\\
\SpM\ar[r]^{\sim} & \Sp}$$
Let $(G^{\cC}_{x,y})^f$ be a fibrant replacement to $G^{\cC}_{x,y}$ in $\SpM$. Then the map $G^{\cC}_{x,y}\to (G^{\cC}_{x,y})^f$, considered in $\Fun_*(\FCW, \Top)$, translates to
$G^{\cC_\infty}_{x, y}\to  \Lin G^{\cC_\infty}_{x, y}$
under the top horizontal map.
By Lemma \ref{lemma: faithful} the last map is equivalent to
$G^{\cC_\infty}_{x, y}\to G^{\SW(\cC_\infty)}_{S(x), S(y)}.$ After applying $\cL$, this map becomes an equivalence, so we have a natural equivalence in $\Sp$
$$G^\cC_{x,y}\simeq G^{\SW(\cC_\infty)}_{S(x), S(y)}\simeq \Hom_{\SW(\cC_\infty)}(S(x),S(y)).$$
By Proposition \ref{p:stab ind} we are done.
\end{proof}

\section{The $\infty$-category of noncommutative CW-spectra}\label{s:NCWSp}
In this section we define the $\infty$-category of noncommutative CW-spectra $\NSp$. The suspension spectra of matrix algebras form a set of compact generators of $\NSp$. We denote by $\cM$ the full spectral subcategory of $\NSp$ spanned by this set of generators (see Proposition \ref{proposition: enrichment}). Thus $\cM$ is an $\Sp$-enriched $\infty$-category. We prove our main theorem: there is an equivalence of $\infty$-categories between $\NSp$ and the category of spectral presheaves on $\cM$. We give two versions and two independent proofs of this result. One version is formulated fully in the language of enriched $\infty$-categories, using Hinich's theory. In the second approach, we first define a strict version of $\cM$, denoted $\cM_s$, which is a category strictly enriched in $\SpM$. We then prove that $\NSp$ is modelled by a Quillen model category of $\SpM$-valued presheaves on $\cM_s$. Finally we prove that our two models of $\cM$ are equivalent, in the sense that $\cM$ is equivalent to the enriched $\infty$-localization of $\cM_s$ (see Definition \ref{d:enriched localization}).

Let us proceed with the definition of $\NSp$. Recall that in Section \ref{ss:complexes} we defined the $\infty$-category of {\it finite} noncommutative CW-complexes and denoted it by $\FNCW$. We then defined the $\infty$-category of {\it all} noncommutative CW-complexes by the formula
$\NCW:=\Ind(\FNCW).$
We now define the $\infty$-category of noncommutative CW-spectra to be
$$\NSp:=\Sp(\NCW).$$
By the results in Section \ref{ss:stab} we know that $\NSp$ is a presentable stable $\infty$-category.
In particular, $\NSp$ is naturally left-tensored over spectra.
By \cite[Corollary 4.8.2.19]{Lur2} the monoidal structure on $\NCW$ induces a closed symmetric monoidal structure on $\NSp$, such that $\Sigma^\infty_{\NC}:\NCW\lrar\NSp$ is symmetric monoidal.

Recall that $M_n$ is the algebra of $n\times n$ matrices over $\mathbb{C}$. Since the set of objects $\{M_i\mid i\in \NN\}$ generates $\FNCW$ under finite colimits, it follows by Corollary~\ref{cor: generates} that $M:=\{\Sigma^\infty_{\NC} M_i\mid i\in \NN\}$ generates $\NSp$ under small colimits and desuspensions.
The following is one of the main definitions of the paper:
\begin{define}\label{definition: M}
Let $\cM$ be the full $\Sp$-enriched subcategory of $\NSp$ spanned by the spectra $\{\Sigma^\infty_{\NC} M_i\mid i\in \NN\}$.
\end{define}
Since $\cM$ is closed under the monoidal product in $\NSp$, the following theorem is a special case of Theorem \ref{t:modules monoidal}:
\begin{thm}\label{theorem: main presentation}
The $\Sp$-enriched category $\cM$ acquires a canonical symmetric monoidal structure, the category of presheaves $P_\Sp (\cM)$ acquires a canonical symmetric monoidal left $\Sp$-tensored structure and we have a natural symmetric monoidal left $\Sp$-tensored functor
$$P_{\Sp}(\cM)\xrightarrow{\sim}\NSp,$$
which is an equivalence of the underlying $\infty$-categories and sends each representable presheaf $Y(\Sigma^{\infty}_{\NC}M_n) \in P_{\Sp}(\cC)$ to $\Sigma^{\infty}_{\NC}M_n$.
\end{thm}


\subsubsection{Strictification of $\cM$}

In this subsection we give a strict model for the category $\cM$ as a monodial spectrally enriched category, as well as a strict version of Theorem \ref{theorem: main presentation}.


In the context of Section \ref{ss: spectral topological}, let us consider the example $\cC=\FNCW$, considered as a topological category. Then $\FNCW$ is a pointed topological category, tensored closed over $\FCW$ (see Definition \ref{d:CW_f tensored}).
As explained in Definition \ref{d:SpM enrichment}, we have a strict enrichment of $\FNCW$ over the model category of spectra $\SpM$ using the functors $G^{\FNCW}_{x, y}$.

The topological category $\FNCW$ has a continuous symmetric monoidal structure induced by tensor product in $\Csep$. The spectral enrichment respects the monoidal structure, in the sense that given objects $x, x_1, y, y_1$, there is a natural transformation
\[
G_{x, y}^{\FNCW}(K)\wedge G_{x_1, y_1}^{\FNCW}(L)\to G_{x\otimes x_1, y\otimes y_1}^{\FNCW}(K\wedge L).
\]
Thus the spectral enrichment of $\FNCW$ is symmetric monoidal.

\begin{define}\label{d:M_s}
  Let $\cM_s$ be the full (strict) $\SpM$-enriched subcategory of $\FNCW$ spanned by $\{M_n\mid n\in\NN\}$. That is, the objects of $\cM_s$ are $\{M_n\mid n\in\NN\}$ and for any $m,n\in\NN$ we have
  $$\Hom_{\cM_s}(M_m,M_n)=G^{\FNCW}_{M_m, M_n}\in\SpM.$$
\end{define}

Since $\cM_s$ is a category enriched over $\SpM$, we can define the strict category of spectral presheaves on $\cM_s$, which we denote by $P_{\SpM}(\cM_s)$, to be the category of enriched functors $\cM_s^{\op}\to \SpM$. We endow $P_{\SpM}(\cM_s)$ with the projective model structure (see, for example, \cite{GM} on the projective model structure in the enriched setting). Since both $\cM_s^{\op}$ and $\SpM$ have a symmetric monoidal structure, the category $P_{\SpM}(\cM_s)$ has a symmetric monoidal structure given by enriched Day convolution turning it into a symmetric monoidal model category.

Consider $\FNCW$ as a category enriched in $\SpM$. There is a canonical strict spectral functor
\begin{equation}\label{eq: strict}
RY\colon \FNCW\to P_{\SpM}(\cM_s),
\end{equation}
which is the composition of the enriched Yoneda embedding $\FNCW\to P_{\SpM}(\FNCW)$ followed by restriction $P_{\SpM}(\FNCW)\to P_{\SpM}(\cM_s)$. It is well-known, and easy to check that $RY$ is lax symmetric monoidal.

We call a map $A\to B$ in $\FNCW$ a weak equivalence if it is a homotopy equivalence in $\FNCW$ considered as a topological category.

\begin{lem}
  The functor $RY$ sends weak equivalences to weak equivalences.
\end{lem}

\begin{proof}
  Let $A\to B$ be a weak equivalence in $\FNCW$. We need to show that $RY(A)\to RY(B)$ is a levelwise weak equivalence in $P_{\SpM}(\cM_s)$. Let $n\geq 1$. We need to show that the induced map $\Hom_{\FNCW}(M_n,A)\to\Hom_{\FNCW}(M_n,B)$ is a weak equivalence in $\SpM$. Since $\SpM$ is a localization of the projective model structure, is enough to show that $\Hom_{\FNCW}(M_n,A)\to\Hom_{\FNCW}(M_n,B)$ is a levelwise weak equivalence in $\Fun_*(\FCW, \Top)$. That is, it is enough to show that for every finite pointed CW-complex $K$,
  $$\Map_{\FNCW}(M_n,K\wedge A)\to\Map_{\FNCW}(M_n,K\wedge B)$$
  is a weak equivalence. Since $\FNCW$ is a topological category, and a weak equivalence in $\FNCW$ is just a homotopy equivalence, we are done.
\end{proof}


By the lemma above, we can apply $\infty$-localization with respect to weak equivalences (see Remark \ref{r:coherent}) to $RY$ and obtain a functor of $\infty$-categories
$$RY_\infty\colon \FNCW_\infty\to P_{\SpM}(\cM_s)_\infty.$$

\begin{lem}
  The $\infty$-category $\FNCW_\infty$ is naturally equivalent to $\FNCW$ defined above as the topological nerve of the topological category $\FNCW$.
\end{lem}

\begin{proof}
  The category $\Csep^{\op}$ (defined in the beginning of Section \ref{ss:complexes}) has the structure of a category of cofibrant objects with the weak equivalences given by the homotopy equivalences and the cofibrations by Schochet cofibrations (see, for instance, in \cite{AG,Uuy}).
We say that a map in $\FNCW$ is a weak equivalence (resp. cofibration) if it is a weak equivalence (resp. cofibration) when regarded as a map in $\Csep^{\op}$. Since $\Csep^{\op}$ is a category of cofibrant objects and $\FNCW \subseteq \Csep^{\op}$ is a full subcategory which is closed under weak equivalences and pushouts along cofibrations it follows that $\FNCW$ inherits a structure of a category of cofibrant objects. In exactly the same way as in \cite[Lemma 7.1.1]{BHH} one can show that the natural map between $\infty$-localizations with respect to weak equivalences:
\[ (\FNCW)_{\infty} \lrar (\Csep^{\op})_\infty \]
is fully faithful.

By \cite[Proposition 3.17]{BJM} we have that the $\infty$-localization of $\Csep^{\op}$  is equivalent to the topological nerve of the topological category structure on $\Csep^{\op}$ described in Section \ref{ss:complexes} (see Remark \ref{r:coherent} and the paragraph before). Since $\FNCW$ is a full topological subcategory of $\Csep^{\op}$, we are done.
\end{proof}

\begin{lem}
The functor $RY_\infty$  preserves finite colimits.
\end{lem}
\begin{proof}
By~\cite[Corollary 4.4.2.5]{Lur1}, it is enough to prove that the functor preserves initial objects and pushout squares.
Since $\FNCW$ is a pointed category, the inital object of $\FNCW$ is also the final object, and the first condition obviously holds.

In both $\FNCW_\infty$ and $P_{\SpM}(\cM_s)_\infty$ pushouts can be calculated as homotopy pushouts in an appropriate structure.
Suppose we have a homotopy pushout diagram in $\FNCW$
\begin{equation}\label{eq: pushout}
\begin{array}{ccc}
y_0 & \to & y_1 \\
\downarrow & & \downarrow \\
y_2 & \to & y_{12}.
\end{array}
\end{equation}
We want to prove that for any $x\in \cM_s$ the induced diagram of functors is a homotopy pushout in the stable model structure
\[
\begin{array}{ccc}
\Map_{\FNCW}(x, -\wedge y_0) & \to & \Map_{\FNCW}(x, -\wedge y_1) \\
\downarrow & & \downarrow \\
\Map_{\FNCW}(x, -\wedge y_2) & \to & \Map_{\FNCW}(x, -\wedge y_{12}).
\end{array}
\]
Since we are working in the stable model structure, a square is a homotopy pushout if and only if it is a homotopy pullback. A square of functors is a homotopy pullback in the stable model structure if the induced square of linearizations is a homotopy pullback. But the linearization of the functor
\[
\Map_{\FNCW}(x, -\wedge y)\colon \FCW\to \Top
\]
evaluated at $K$ is the same as the linearization of the functor
\[
\Map_{\FNCW}(x, K\wedge -)\colon \FNCW\to \Top
\]
evaluated at $y$. Indeed, the two linearizations are given by the equivalent formulas
\[
\hocolim_{n\to \infty}  \Omega^n \Map_{\FNCW}(x, \Sigma^n K\wedge y)= \hocolim_{n\to \infty} \Omega^n \Map_{\FNCW}(x, K\wedge \Sigma^n_{\FNCW} y).
\]
We have been thinking of the functor $\Map_{\FNCW}(x, K\wedge -)\colon \FNCW\to \Top$ as a strict functor, but now let us think of it as a functor between $\infty$-categories by applying the $\infty$-localization. The $\infty$-category $\FNCW$ has finite colimits and a final object. The conditions of~\cite[Lemma 6.1.1.33]{Lur1} are satisfied, and therefore the linearization of this functor really is linear, i.e., takes homotopy pushout squares to homotopy pullback squares. Therefore applying the linearization to the square~\eqref{eq: pushout} yields a homotopy pullback square, which is what we wanted to prove.
\end{proof}

Since $P_{\SpM}(\cM_s)_\infty$ is a stable $\infty$-category, we have, by the lemma above, that $RY_\infty$ extends canonically to a finite-colimit preserving functor \[RY_\infty:\SW(\FNCW)\to P_{\SpM}(\cM_s)_\infty.\] This functor extends, in turn, to an all-small-colimit-preserving functor
\begin{equation}\label{eq: infinity}
RY_\infty\colon \NSp=\Ind(\SW(\FNCW))\to P_{\SpM}(\cM_s)_\infty.
\end{equation}
This functor takes an object $\Sigma^\infty_{\NC} M_n$ to the presheaf represented by $M_n$.

\begin{thm}\label{theorem: main presentation strict}
The functor $RY_\infty\colon \NSp \to P_{\SpM}(\cM_s)_\infty$ is an equivalence of $\infty$-categories.
\end{thm}

\begin{proof}
First let us prove that $RY_\infty$ is fully faithful, that is, that for all objects $x, y$ of $\NSp$ the map of spectral mapping functors
\begin{equation}\label{eq: comparison}
G^{\NSp}_{x, y}\to G^{P_{\SpM}(\cM_s)_\infty}_{RY_\infty(x), RY_\infty(y)}
\end{equation}
is an equivalence. First, consider the case $x, y\in \Sigma^\infty\cM_s$, i.e., $x=\Sigma^\infty_{\NC} M_k, y=\Sigma^\infty_{\NC} M_l$ for some $k, l$. In this case $x, y$ are in the image of the finite suspension functor $\FNCW\to \SW(\FNCW)$. The functor $\SW(\FNCW)\to \NSp$ is fully faithful, so it induces an equivalence
\[
G^{\SW(\FNCW)}_{x, y}\xrightarrow{\simeq} G^{\NSp}_{x, y}.
\]
By Lemma~\ref{lemma: faithful}, the map $G^{\FNCW}_{M_k,M_l} \to G^{\SW(\FNCW)}_{x, y}$ is stabilization.

The functor $RY:\FNCW\to P_{\SpM}(\cM_s)$ when restricted to $\cM_s$ is just the enriched Yoneda embedding of $\cM_s$
$$Y:\cM_s\to P_{\SpM}(\cM_s).$$
Since the unit in $\SpM$ is cofibrant, $RY_\infty(\Sigma^\infty M_k)=Y(M_k)$ is cofibrant in the projective model structure on $P_{\SpM}(\cM_s)$ (see \cite[Theorem 4.32]{GM}). The fibrant replacement in $P_{\SpM}(\cM_s)$ is levelwise, so using the (strict) enriched Yoneda lemma we get
$$G^{P_{\SpM}(\cM_s)_\infty}_{RY_\infty(x), RY_\infty(y)}= \Hom_{P_{\SpM}(\cM_s)_\infty}(Y(M_k),Y(M_l))\simeq$$
$$\Hom_{P_{\SpM}(\cM_s)}(Y(M_k),Y(M_l)^f)\cong (Y(M_l)^f)(M_k)\simeq
Y(M_l)(M_k)^f =(G^{\FNCW}_{M_k,M_l})^f.$$
But the map $G^{\FNCW}_{M_k,M_l}\to (G^{\FNCW}_{M_k,M_l})^f$ translates to the stabilization
$G^{\FNCW}_{M_k, M_l}\to  \Lin G^{\FNCW}_{M_k, M_l}$
under $\infty$-localization (see the proof of Theorem \ref{t:mapping}).
Thus the map
\[
G^{\FNCW}_{M_k,M_l} \to  G^{P_{\SpM}(\cM_s)_\infty}_{RY_\infty(x), RY_\infty(y)}
\]
is also the stabilization. By the uniqueness of the stabilization map, we get that the map~\eqref{eq: comparison}
is an equivalence in the case when $x, y$ are suspension spectra of matrix algebras.

Next, let $x$ be a fixed suspension spectrum of a matrix algebra, but let $y$ vary. We may consider the functor $y\mapsto G^{\NSp}_{x, y}$ as a functor $\NSp\to \Sp$. This functor preserves all small colimits, because $x$ is compact in $\NSp$ and both $\NSp$ and $\Sp$ are stable. Similarly, the functor $y\mapsto G^{P_{\SpM}(\cM_s)_\infty}_{RY_\infty(x), RY_\infty(y)}$ is also a functor $\NSp \to \Sp$ that preserves small colimits. It follows that the category of objects $y$ for which the map~\eqref{eq: comparison} is an equivalence is closed under colimits and also desuspensions. Since this category contains $\cM_s$, it is all of $\NSp$.

Now fix $y$, and consider the functors $x\mapsto G^{\NSp}_{x, y}, G^{P_{\SpM}(\cM_s)_\infty}_{RY_\infty(x), RY_\infty(y)}$ as contravariant functors from $\NSp$ to spectra. Since both functors take small colimits to limits, a similar argument shows that this map is an equivalence for all $x\in \NSp$.

We have shown that the functor $RY_\infty\colon \NSp\to P_{\SpM}(\cM_s)_\infty$ is fully faithful and also preserves all small colimits, so it is a left adjoint. It follows that the image of $RY_\infty$  is closed under small colimits. Since the image contains the representable presheaves, $RY_\infty$ is essentially surjective. It follows that $RY_\infty$ is an equivalence of categories. \end{proof}

Thus $RY_\infty$ is an explicit monoidal model for the inverse of the equivalence given by Theorem~\ref{theorem: main presentation}. This also allows us to show that $\cM_s$ is indeed a strictification of $\cM$ from Definition \ref{definition: M}.

\begin{thm}
  We have a natural equivalence
  $(\cM_s)_\infty\simeq\cM$
  between the enriched $\infty$-localization of $\cM_s$ and $\cM$.
\end{thm}

\begin{proof}
By definition, $\cM_s$ is an $\SpM$-enriched category, whose set of objects is the set of natural numbers $\mathbb N$. Let $\SpM\mathbb N$-$\Cat$ be the category of all $\SpM$-enriched categories, whose set of objects is $\mathbb N$, and whose morphisms are functors that are the identity on objects. This category has a Quillen model structure, where fibrations and weak equivalences are defined levelwise, and where cofibrant objects are levelwise cofibrant~\cite[Proposition 6.3]{ScSh2}.
Thus, $\cM_s$ is an object in $\SpM\mathbb N$-$\Cat$. Let $\cM_s\to \cM_s^f$ be a fibrant replacement of $\cM_s$ in $\SpM\mathbb N$-$\Cat$. We get a Quillen adjunction
$$\LKan_i:P_{\SpM}(\cM_s)\rightleftarrows P_{\SpM}(\cM_s^f):i^*.$$
It follows from~\cite[Proposition 2.4]{GM} that
this adjunction is a Quillen equivalence. To give a little more detail, it follows from the general result of Guillou and May that it is enough to show that for every cofibrant object $M$ of $\SpM$ and every two objects $x, y$ of $\cM_s$, the following induced map is an equivalence
\[
M\wedge \Hom_{\cM_s}(x, y)\to M\wedge \Hom_{\cM_s^f}(x, y),
\]
where $\Hom(-, -)$ denotes the spectral mapping object. The map
$$\Hom_{\cM_s}(x, y)\to \Map_{\cM_s^f}(x, y)$$
is an equivalence by definition of $M_s^f$. It follows by~\cite[Proposition 12.3]{MMSS} that the induced map is an equivalence for all cofibrant $M$.

Applying $\infty$-localization we obtain an equivalence
$$P_{\SpM}(\cM_s)_\infty\xrightarrow{\sim} P_{\SpM}(\cM_s^f)_\infty.$$
Now, $P_{\SpM}(\cM_s^f)$ is an $\SpM$-model category, and it is Quillen equivalent to a combinatorial model category. The enriched Yoneda embedding
$$Y:\cM_s^f\to P_{\SpM}(\cM_s^f)$$
is a fully faithful $\SpM$-enriched functor and it clearly lands in the fibrant cofibrant objects. Thus, by Corollary \ref{c:fully faithful}, the $\Sp$-functor
$$Y_\infty:(\cM_s^f)_\infty\to P_{\SpM}(\cM_s^f)_\infty$$
is $\Sp$-fully faithful. Note that this claim is made with respect to the action of $\Sp$ on $P_{\SpM}(\cM_s^f)_\infty$ induced by the structure of $P_{\SpM}(\cM_s^f)$ as an $\SpM$-model category.
Since $\Sp$ is a mode in the sense of \cite[Section 5]{CSY}, this action coincides with the canonical action of $\Sp$ on $P_{\SpM}(\cM_s^f)_\infty$ as a presentable stable $\infty$-category.

Now, using Theorem \ref{theorem: main presentation strict}, we have the following composition
$$(\cM_s^f)_\infty\xrightarrow{Y_\infty}  P_{\SpM}(\cM_s^f)_\infty\xrightarrow{\sim} P_{\SpM}(\cM_s)_\infty\xrightarrow{\sim} \NSp.$$
We get a fully faithful $\Sp$-enriched functor $(\cM_s^f)_\infty\to\NSp$, with essential image $\cM$, so that $(\cM_s^f)_\infty\simeq\cM$.
\end{proof}


\end{document}